\documentclass[preprint]{imsart}
\usepackage{stmaryrd}
\usepackage{multirow,mathtools,latexsym, enumerate, enumitem, amsmath, amsthm, amssymb,hyperref, pdftricks,tikz,framed,color,enumitem,enumerate,mleftright}
\usetikzlibrary{arrows.meta, decorations.markings,math} %needed tikz libraries

\pdfoutput=1

\startlocaldefs

\def\x{{x}}
\def\y{{y}}
\def\Y_#1{y_{\!#1}}

\def\E{\mathbb{E}}
\def\P{\mathbb{P}}
\def\L{L}
\def\R{\mathbb{R}}

\def\H{\mathcal{H}}
\def\TV{\mathrm{TV}}

\def\tr{\operatorname{trace}}

\newcommand{\bph}[1]{\boldsymbol{#1}} % for vectors

\newcommand{\rn}[1]{\Romanbar{#1}}

\renewcommand{\tilde}{\widetilde} 
\newcommand{\norm}[1]{\left\lvert #1 \right\rvert} % H^s norm
\newcommand{\snorm}[1]{\left\lvert #1 \right\rvert_s} % H^s norm
\newcommand{\anorm}[1]{{\left\vert\kern-0.25ex\left\vert\kern-0.25ex\left\vert #1 
    \right\vert\kern-0.25ex\right\vert\kern-0.25ex\right\vert_{\alpha}}}
\endlocaldefs
\numberwithin{equation}{section}

\newtheorem{theorem}{Theorem}[section]
\newtheorem{corollary}[theorem]{Corollary}
\newtheorem{lemma}[theorem]{Lemma}

\newtheorem{remark}[theorem]{Remark}

\newtheorem{assumption}[theorem]{Assumption}
\usepackage{tikz}
\usepackage{romanbar}

\theoremstyle{definition}

\newtheorem{algorithm}{Algorithm}[section]
\newtheorem{example}{Example}[section]

\theoremstyle{remark}

\definecolor{myGreen}{rgb}{0.0, 0.5, 0.0}
\begin{document} 

\title{Two-scale coupling for preconditioned Hamiltonian Monte Carlo in\\ infinite dimensions}
% Two-scale coupling for Hilbert-space HMC?
\runtitle{Two-scale Coupling for HMC}

 \begin{aug}
 \author{\fnms{Nawaf} \snm{Bou-Rabee}\thanksref{t1,m1}\ead[label=e1]{nawaf.bourabee@rutgers.edu}}
 \and
 \author{\fnms{Andreas}
 \snm{Eberle}\thanksref{t2,m2}\ead[label=e2]{eberle@uni-bonn.de}}
 
 \address{\thanksmark{m1}Department of Mathematical Sciences \\ Rutgers University Camden \\ 311 N 5th Street \\ Camden, NJ 08102, USA \\ \printead{e1}}
 \address{\thanksmark{m2}Institut f\"{u}r Angewandte Mathematik \\  Universit\"{a}t Bonn \\ Endenicher Allee 60 \\  53115 Bonn, Germany \\ \printead{e2}}
 
 \thankstext{t1}{N.~B.-R. was supported by the National Science Foundation under Grant No.~DMS-1816378 and the Alexander von Humboldt Foundation.}
 \thankstext{t2}{A.E. has been supported by the Hausdorff Center for Mathematics. Gef\"ordert durch die Deutsche Forschungsgemeinschaft (DFG) im Rahmen der Exzellenzstrategie des Bundes und der L\"ander - GZ 2047/1, Projekt-ID 390685813.}
 
 \runauthor{N.~Bou-Rabee, A.~Eberle}
 \end{aug}

 \begin{abstract}
 We derive non-asymptotic quantitative bounds for convergence to 
 equilibrium of the exact preconditioned Hamiltonian Monte Carlo algorithm (pHMC) 
 on a Hilbert space. As a consequence, explicit and dimension-free bounds for pHMC applied to high-dimensional distributions arising in transition path sampling and path integral molecular dynamics are given.  Global convexity of the underlying potential energies is not required. Our results are based on a two-scale coupling which 
 is contractive in a carefully designed distance. 
 % Our main new tool is a two-scale coupling that is contractive in a carefully designed distance.
 \end{abstract}


% Primary:
% 60J05  Discrete-time Markov processes on general state spaces 
% Secondary:
% 60B12   Limit theorems for vector-valued random variables (infinite-dimensional case)
% 60J22  Computational methods in Markov chains 
% 65P10  Hamiltonian systems including symplectic integrators
% 65C05 Monte Carlo Methods
% 65C40 Computational Markov Chains


\maketitle

%%%%%%%%%%%%%%%%%%%%%%%%%%%%%%%%%%%%%%%%%%%%%
%%%%%%%%%%%%%%%%%%%%%%%%%%%%%%%%%%%%%%%%%%%%%

\section{Introduction}

% motivation for preconditioned MCMC methods

Hamiltonian or Hybrid Monte Carlo (HMC) methods are a class of Markov Chain Monte
Carlo (MCMC) methods originating in statistical physics \cite{DuKePeRo1987} which have become 
increasingly popular in various application areas \cite{Li2008,mathias2010free,neal2011mcmc,BoSaActaN2018,prokhorenko2018large}. Their success is in particular due to empirically observed 
convergence acceleration compared to more traditional, random-walk based methods. 
The basic idea in HMC is to define an MCMC method with the help of an artificial Hamiltonian dynamics whose {\em only} purpose is to accelerate convergence to equilibrium. This Hamiltonian dynamics is designed to leave invariant a product of the target measure and a fictitious Gaussian measure in an artifical velocity variable. 
First rigorous theoretical results supporting the empirical evidence have
only been established recently. In particular, geometric ergodicity has been verified
in \cite{BoSa2016,DurmusMoulinesSaksman,Livingstone}, and quantitative convergence bounds have
been derived in the strongly convex case in \cite{mangoubi2017rapid},
and under more general assumptions in \cite{BoEbZi2018}, both by applying coupling methods. \smallskip

Since many applications are high dimensional, a key issue is to understand the 
dependence of the convergence bounds on the dimension. Here, we study the problem of dimension dependence for a special class of models that is relevant for several important applications including Path Integral Molecular Dynamics (PIMD) \cite{ChandlerWolynes,
  Miller2005,Miller2005a,Craig2004,Craig2005,Craig2005a,Habershon2008,
  Habershon2013,Lu2018,KoBoMi2019}, Transition Path
Sampling (TPS) \cite{ReVa2005,pinski2010transition,Bolhuis2002,Miller2007}, and Bayesian inverse problems \cite{kaipio2005statistical,stuart2010inverse,dashti2017bayesian,borggaard2018bayesian}. 
%In these applications, the simple `unadjusted' version of HMC is commonly implemented; more on this terminology below.
For the class of models we consider, a corresponding HMC Markov chain
relying on a preconditioned Hamiltonian dynamics
can be defined directly on the infinite dimensional state space
\cite{BePiSaSt2011}. This suggests that one might hope for
dimension-free convergence bounds for the corresponding Markov chains
on finite-dimensional discretizations of the state space.
Corresponding dimension-free convergence rates to equilibrium have been
established for the preconditioned Crank-Nicholson (pCN) algorithm
\cite{hairer2014spectral} and for the Metropolis-adjusted Langevin
algorithm (MALA) \cite{eberle2014}, but a corresponding result for 
HMC is not known so far.\smallskip

The goal of this paper is to fill this gap. To this end we extend the
coupling approach developed for HMC in the finite dimensional case in \cite{BoEbZi2018}, and combine it with a two-scale coupling approach
for stochastic dynamics on infinite dimensional Hilbert spaces that
originates in \cite{Ma1998,mattingly2003recent,hairer2002exponential,hairer2008spectral} and has been further developed in \cite{Zi2017}. The splitting 
into ``low modes'' and ``high modes'' in the two-scale coupling can be traced back to contraction results for the stochastic Navier-Stokes equations \cite{Ma1998}, and analogous results in the deterministic setting \cite{foias1967comportement}; see \cite{mattingly2003recent} for a detailed review.  
\smallskip

Our object of study is the \emph{exact preconditioned HMC algorithm} (pHMC) with fixed 
durations on a Hilbert space, i.e., 
%We consider pHMC as a Markov chain on a Hilbert space with fixed durations where 
the (preconditioned) Hamiltonian dynamics is exactly integrated (or, in practical terms, the integration is carried out with 
very small step sizes). 
%such a small value of the time step size which ensures that essentially all steps of the chain result in acceptance), 
Here, preconditioning corresponds to an appropriate choice of the 
kinetic energy 
%  For target measures on Hilbert spaces which have a density with respect to a Gaussian reference measure, the kinetic energy of choice 
  which involves picking the mass operator equal to the stiffness operator (or inverse covariance) associated to the Gaussian reference measure of the target probability measure.  This choice of kinetic energy 
  ensures that the corresponding
  pHMC algorithm is more amenable to numerical approximation and Metropolis-adjustment than HMC without preconditioning \cite{BePiSaSt2011, BoSaActaN2018}.\smallskip
  
We prove that the transition kernel of the Markov chain
induced by the pHMC algorithm
is contracting in a suitable Wasserstein/Kantorovich metric with an explicit contraction rate.  This rate depends on the duration of the Hamiltonian flow, the eigenvalues of the covariance operator of the Gaussian reference measure, and the regularity of the preconditioned Hamiltonian dynamics. The results are given in a more general setting that includes pHMC as a special case, and also covers other types of dynamics and preconditioning strategies. As a 
consequence of our general results, we derive dimension-free  bounds for pHMC applied to  finite-dimensional approximations arising in TPS and PIMD.\smallskip

Before stating our results in detail, we conclude with a brief outlook.
The results below apply only to pHMC with exact integration of the Hamiltonian
dynamics.
In practice, the Hamiltonian dynamics is numerically approximated, to obtain {\em numerical} versions of pHMC that are implementable on a computer. The time integrator of choice for pHMC is the symmetric splitting integrator introduced in \cite{BePiSaSt2011}.  Unlike other splittings for the Hamiltionian dynamics, this approximation has an acceptance rate that is uniform with respect to the spatial step size associated with the discretization of the Hilbert space \cite[\S 8]{BoSaActaN2018}. Time discretization
creates a bias in the invariant measure that can be avoided by a 
Metropolis adjustment \cite{tierney1998note,BoSaActaN2018}. We would expect that for unadjusted numerical HMC 
based on the integrator proposed in \cite{BePiSaSt2011}, 
similar contraction results as stated below hold if the
time step size is chosen sufficiently small (but independently of the 
dimension). Under additional regularity assumptions, one could also
hope for dimension free bounds for the Metropolis adjusted version.
First steps in this direction are carried out in \cite[\S 2.5.4]{BoEbZi2018} in the finite dimensional case, and in \cite[\S 4]{PidstrigachMT}
in a strongly convex infinite dimensional case,
but a full study in the general case would be lengthy and go beyond 
the scope of the current work.\smallskip

%Furthermore, to correct the  bias introduced by time discretization error, these numerical approximations can be Metropolis-adjusted, to obtain {\em adjusted} versions of these approximations, which ensure that the approximation, as a deterministic map, leaves invariant the Boltzmann-Gibbs distribution \cite{tierney1998note,BoSaActaN2018}.   The algorithm without adjustment is called {\em unadjusted} HMC.  Thus, in addition to the duration of the Hamiltonian flow, another key parameter in the pHMC method is the time step size used to generate this numerical solution.   \smallskip

%Analogous to the finite-dimensional case \cite[see \S 2.5.4]{BoEbZi2018}, we expect the results for pHMC to carry over for the unadjusted and adjusted versions of this method. 

Alternatively to preconditioning, it is also possible (but more tricky) to
implement
non-preconditioned HMC, which corresponds to injecting white noise in the 
velocity variable. In this case, 
the corresponding Hamiltonian dynamics is highly oscillatory in high modes \cite{petzold_jay_yen_1997}.
Therefore, convergence bounds for exact HMC 
without preconditioning on an infinite dimensional Hilbert space can be
expected to hold only if the durations are randomized \cite{neal2011mcmc},
and in numerical implementations, strongly stable integrators \cite{KoBoMi2019} have to be used in order to be able to choose the step size
independently of the dimension. Furthermore, scaling limit results show
that for Metropolis adjusted HMC applied to i.i.d.\ product measures
on high dimensional state spaces, the step size has to be chosen of 
order $O(d^{-1/4})$ to avoid degeneracy of the acceptance probabilities
\cite{gupta1988tuning,kennedy1991acceptances,BePiRoSaSt2013}.

\smallskip

We now state our main results in Section \ref{sec:main_results},
and consider applications to TPS and PIMD in Section \ref{sec:applications}.
The remaining sections contain the proofs of all results.

\section{Main results} \label{sec:main_results}

%\subsection{Gaussian reference measure} \label{sec:main_results:Gauss} 

Let $\H$ be a separable and real Hilbert space with inner product $\langle \cdot, \cdot \rangle$ and norm $\norm{\cdot}$. Let $\mathcal{C}: \H \to \H$ be a positive compact symmetric linear operator.
By the spectral theorem, the eigenfunctions $\{ e_i \}_{i \in \mathbb{N}}$ of
  $\mathcal{C}$ form a complete orthonormal basis of $\H$ with corresponding eigenvalues 
  $\{ \lambda_i \}_{i \in \mathbb{N}}$ which we arrange in descending order, i.e., $\lambda_1 \ge \lambda_2 \ge \cdots$.   The positivity condition means that $\lambda_j>0$ for all $j \in \mathbb{N}$, and by compactness, if $\text{dim}(\H )=\infty$ then $\lim_{j\to\infty}\lambda_j=0$.  Any function $x \in \H$ can be represented in spectral coordinates by the  expansion \begin{equation}
  x = \sum x_j e_j \quad \text{where $x_j := \langle x, e_j \rangle$.}
  \end{equation}
Moreover, for all $s\in\mathbb R$, the operator $\mathcal{C}^s$ is defined via the spectral decomposition of $\mathcal{C}$. We introduce the family of inner products and norms given by \begin{equation} \label{eq:Hs_norm}
\langle x,y\rangle_s:= \langle  x, \mathcal{C}^{-s} y \rangle = \langle \mathcal{C}^{-s/2} x, \mathcal{C}^{-s/2} y \rangle \;, \quad 
\snorm{x}:=\langle x,x\rangle_s^{1/2} 
\end{equation}
for $x,y\in \H^s$. Here for $s\ge 0$, $\H^s$ denotes the Hilbert space consisting of all
$ x \in \H $ with $\snorm{x} < \infty $,
whereas for $s<0$, $\H^s$ is the completion of $\H$ w.r.t.\ $\snorm{x}$.
Note that $\H=\H^0$, and for $s>0$, $\H^s \subset \H \subset \H^{-s}$. Furthermore,
the linear operator $\mathcal{C}$ restricts or extends (depending on whether $s>0$ or $s<0$) to a linear isometry from $\H^s$ to $\H^{s+2}$ which will again be denoted by $\mathcal{C}$.\medskip

We will now introduce the pHMC method for 
approximate sampling from a probability measure $\mu$ that has a density 
w.r.t.\ a Gaussian measure $\mu_0$ on one of the Hilbert spaces $\H^s$.
Afterwards, in Section \ref{sec:main_results:assumptions}, we will introduce
a more general family of Markov chains on Hilbert spaces that includes the Markov chain associated to
pHMC as a special case. In Section \ref{sec:main_results:couplings},
we introduce a new coupling for these Markov chains that combines 
ideas from \cite{BoEbZi2018} and \cite{Zi2017}. Then in Sections \ref{sec:main_results:contraction} and \ref{sec:QB}, we state our main 
contraction result for these couplings, and derive
quantitative error bounds.

\subsection{Exact Preconditioned Hamiltonian Monte Carlo} \label{sec:main_results:pHMC}

Let $\mu_0=\mathcal{N}(0, \mathcal{C})$ denote the centered Gaussian measure whose covariance operator w.r.t.\ the inner product $\langle \cdot, \cdot \rangle$ is $\mathcal{C}$ \cite{bogachev1998gaussian}. If $\mathcal{C}$ is trace class then $\mu_0$ is
supported on $\H$. More generally,
we fix $s\in (-\infty ,1)$ and assume that $\mu_0$ is supported on the corresponding Hilbert space $\H^s$. This is ensured by the following assumption:
\begin{assumption} 
\label{A456}
%There exists an $s \in \mathbb{R}$ such that 
The operator $\mathcal{C}^{1-s}$ is trace class, i.e., $$\tr(\mathcal{C}^{1-s}) = \sum_{j=1}^{\infty} \lambda_j^{1-s} <\infty .$$ \end{assumption} 
A realization $\xi$ from $\mu_0$ can be generated using the expansion 
\[
\xi = \sum_{j=1}^{\infty} \sqrt{\lambda_j} \rho_j e_j \;, \quad \{ \rho_i \} \overset{\text{i.i.d.}}{\sim} \mathcal{N}(0,1) \;.
\] 
For $\xi \sim \mu_0$, Assumption \ref{A456} implies $\E \snorm{\xi}^2 = \tr(\mathcal{C}^{1-s})   < \infty$, and thus, $\xi$ is indeed a Gaussian random variable on $\H^s$.
\begin{remark} To avoid confusion, we stress that the covariance operator of a Gaussian measure
is a non-intrinsic object that depends on the choice of an inner product. In particular,
the covariance operator of $\mu_0$ w.r.t.\ the $\H^s$ inner product is $\mathcal{C}^{1-s}$.  Nonetheless, in what follows, we always define the covariance operator with respect to the $\H$ inner product, and in this sense, the measure $\mu_0$ has covariance operator $\mathcal{C}$. 
% note that $\mathcal{N}_{\H}(0,\mathcal{C}) = \mathcal{N}_{\H^s}(0,\mathcal{C}^{1-s})$
\end{remark}

\emph{Exact preconditioned Hamiltonian Monte Carlo (pHMC)} is an MCMC method for approximate sampling
from probability distributions on a Hilbert space that have the general form \begin{equation} \label{eq:mu}
\mu(dx) \propto \exp(- U(x) ) \mu_0(dx) \;, \quad \mu_0 = \mathcal{N}(0, \mathcal{C}) \;,
\end{equation}
where $U$ is a function on a Hilbert space on which the Gaussian measure $\mu_0$
is supported.
%potential energy and $\mathcal{N}(0, \mathcal{C})$ denotes a centered Gaussian reference measure with covariance operator $\mathcal{C}$.
%\textcolor{green}{Should we mention the phase space measure here at all?}
%of the form \eqref{eq:mu} or \begin{equation} \label{eq:hatmu}
%\hat{\mu}(dx \; dv) \propto \exp(- U(x) ) \mathcal{N}(0, \mathcal{C})(dx) \mathcal{N}(0, \mathcal{C})(dv) \;.
%\end{equation}
%Note that the $x$-marginal of \eqref{eq:hatmu} is given by \eqref{eq:mu}.  
%In this paper, we regard pHMC as a Markov chain on the configuration space $\H^s$ (not on the phase space $\H^s\times\H^s$). 
The pHMC method generates a Markov chain on this Hilbert space with transition step 
%from $x \in \H^s$ is given by
 \begin{eqnarray}
 x\mapsto X'(x) \quad \text{where} \quad	\label{*} X'(x) &=& q_T(x,\xi) \;.
 	\end{eqnarray}
Here $\xi\sim\mathcal N(0,\mathcal C)$, and the duration
$T:\Omega\rightarrow \R_+$ is in general an independent random variable with a given distribution
$\nu$ (e.g.\ $\nu=\delta_r$ or $\nu=\text{Exp}(\lambda^{-1})$).   We will only consider the case where $T\in (0,\infty)$ is a given deterministic constant. Moreover,
\begin{eqnarray*}
	\phi_t(x,v) &=& \left(q_t(x,v), v_t(x,v)\right)\qquad ( t\in [0,\infty ))
\end{eqnarray*}
is the exact flow of the Hilbert space valued ODE given by
\begin{eqnarray}\label{eq:dynamics}
	\frac{d}{dt} q_t  \ = \ v_t, \quad \frac{d}{dt} v_t \ = \ -q_t-\mathcal CDU(q_t),\quad
	\left(q_0(x,v),v_0(x,v)\right) \ = \ (x,v).
\end{eqnarray}
Formally, \eqref{eq:dynamics} is a preconditioned Hamiltonian dynamics for the Hamiltonian $$H(x,v)=U(x)+
\langle x, \mathcal{C}^{-1}x\rangle/2+\langle v, \mathcal{C}^{-1}v\rangle /2 ,$$
where the covariance operator $\mathcal{C}$ is used for preconditioning.
%The corresponding Markov chain with transition step determined by \eqref{*} and \eqref{eq:dynamics} is called \emph{pHMC}. 
A key property of \eqref{eq:dynamics} is that it leaves invariant the probability measure
\begin{equation} \label{eq:hatmu}
\hat{\mu}(dx \; dv)\ \propto\ \exp(- U(x) )\, \mathcal{N}(0, \mathcal{C})(dx)\, \mathcal{N}(0, \mathcal{C})(dv) \;,
\end{equation}
on phase space, and in turn, this implies that the transition kernel of pHMC defined by $\pi(x,B) = \P[X'(x) \in B]$ leaves $\mu$ in \eqref{eq:mu} invariant \cite{BePiSaSt2011}. 
\medskip

Below, our key assumption in this setup will be that $U$
is a gradient Lipschitz function on the Hilbert space $\H^s$
where the reference measure $\mu_0$ is supported:
%Throughout this paper we assume that the target measure $\mu$ has a density with respect to the Gaussian reference measure $\mu_0$ on $\H^s$.  

 \begin{assumption}\label{MU123}
The target measure $\mu$ 
is a probability measure on $\H^s$ that is absolutely continuous with respect to $\mu_0$. The relative density is proportional to
$\exp(-U)$ where $U:\mathcal H^s\to [0,\infty )$ is a Fr\`echet differentiable function
satisfying the gradient Lipschitz condition
$$\left|\partial_hU(x)-\partial_hU(y)\right|\ \le\ L_g\, \left|x-y\right|_s\, \left|h\right|_s\qquad\text{for all }x,y,h\in\H^s$$
for some finite positive constant $L_g$.
\end{assumption}

In Assumption \ref{MU123}, $\partial_h U$ denotes the directional derivative of $U$ in direction $h$. We also use the notation $DU$ to denote the differential of $U$, i.e., $(DU)(x)$ is the linear functional on $\H^s$ defined by
$(DU)(x)[h]=(\partial_hU)(x)$. Identifying the dual space of $\H^s$ with $\H^{-s}$, Assumption \ref{MU123} shows that we can view $D U$ as a Lipschitz continuous function from $\H^s$ to $\H^{-s}$, i.e.,
\begin{equation}\label{eq:gradlip}
    \left|DU(x)-DU(y)\right|_{-s}\ \le\ L_g \, \left|x-y\right|_s\qquad\text{for all }x,y\in\H^s.
\end{equation}
Recalling that $\mathcal C$ is an isometry from $\H^{-s}$ to $\H^{2-s}$,
and $\H^{2-s}$ is continuously embedded into $\H^s$ for $s<1$, we see that Assumption \ref{MU123} implies that the drift function
\begin{equation} \label{eq:b}
b(x):= - x- \mathcal{C} D U(x) 
\end{equation}
occurring in \eqref{eq:dynamics} is a Lipschitz continuous map from $\H^s$ to $\H^s$.

\begin{remark}
 The global Lipschitz condition in Assumption~\ref{MU123} is principally the same as Condition~3.2 in \cite{BePiSaSt2011} except here the domain $\H^s$ of the potential energy is defined in terms of the covariance operator itself rather than in terms of an auxiliary operator with related eigenfunctions and eigenvalues.
\end{remark}

%Hence, the Radon-Nikodym derivative of $\mu$ with respect to $\mu_0$, i.e., $\frac{d \mu}{d \mu_0}(x) \propto \exp(-U(x))$, is a well-defined probability density on $\H^s$, and in particular, $\exp(-U(x))$ is normalizable with respect to $\mu_0$.  The normalizable result is typically established by invoking the Fernique Theorem; see, e.g, \cite{hairer2009introduction,stuart2010inverse}.

%For any $f: \H^s \to \mathbb{R}$, let $\grad_s f: \H^{s} \to \H^{-s}$ denote the gradient of $f$ with respect to the inner product on $\H^s$, i.e., for all $v \in \H^s$ \[
%\frac{d}{d\epsilon}  f(x+\epsilon v) \bigg\rvert_{\epsilon=0} = \langle \grad_s f(x), v \rangle_s  \quad \text{and let} \quad \nabla f := \grad_0 f \;.
%\] 
%Thus, $\grad_s f = \mathcal{C}^{-s} \nabla f$.   

%In applications \cite{stuart2010inverse,hairer2011signal}, the potential energy $U$ is defined on a sub- or super-space of $\H$ where the Gaussian reference measure $\mu_0$ is supported.  To allow this possibility, for any $s \in \mathbb{R}$, introduce the space  

\subsection{General setting} \label{sec:main_results:assumptions} 

We now introduce a more general setup that includes the Markov chain induced by pHMC as a special case. We fix $s\in\mathbb R$, and we assume that $b: \H^{s} \to \H^{s}$
is a Lipschitz continuous function. Let $
	\phi_t(x,v) = \left(q_t(x,v), v_t(x,v)\right)$
denote the exact flow of the Hilbert space valued ODE given by
\begin{eqnarray}\label{eq:dynamicsgeneral}
	\frac{d}{dt} q_t  \ = \ v_t, \quad \frac{d}{dt} v_t \ = \ b(q_t),\quad
	\left(q_0(x,v),v_0(x,v)\right) \ = \ (x,v).
\end{eqnarray}
As above, we fix a constant duration $T\in (0,\infty )$ and consider the Markov chain on $\H^s$ with transition step 
 \begin{equation}
    \label{eq:transitionstep}
 x\ \mapsto\ X'(x)\ :=\ q_T(x,\xi) \;,\qquad \xi\sim\mathcal N(0, \tilde{\mathcal{C}})
 \end{equation}
 where $\tilde{\mathcal{C}}$ is a linear operator on $\H$ with the same eigenfunctions as $\mathcal C$.
 
 \begin{assumption} 
\label{C123}
$\tilde{\mathcal{C}}: \H \to \H$ is a symmetric linear operator with eigenfunctions $\{e_i\}_{i \in \mathbb{N}}$ and corresponding eigenvalues $\{\tilde{\lambda}_i\}_{i \in \mathbb{N}}$. Moreover, 
the operator $\tilde{\mathcal{C}} \mathcal{C}^{-s}$ is trace class, i.e.,  $$\tr(\tilde{\mathcal{C}} \mathcal{C}^{-s}) = \sum_{j=1}^{\infty} \tilde{\lambda}_j \lambda_j^{-s} <\infty .$$  
\end{assumption}
 
 For $\xi\sim\mathcal N(0, \tilde{\mathcal{C}})$, this implies $\E\snorm{\xi}^2 = \tr(\tilde{\mathcal{C}} \mathcal{C}^{-s}) < \infty$, and thus, $\xi$ in \eqref{eq:transitionstep} is a Gaussian random variable on $\H^s$.  Let $\pi(x,B) = \P[X'(x) \in B]$ denote the corresponding transition kernel.  In particular, in the case where $b$ is given by \eqref{eq:b} and $\tilde C=C$, we recover the Markov chain associated to pHMC.\medskip

Our main result rests on the assumption that the Hilbert space $\H^s$ can be split
into a finite dimensional subspace $\H^{s,\ell}$ (``the low modes'') and its orthogonal
complement $\H^{s,h}$ (``the high modes'') such that $b(x)$ is close to
a linear map on $\H^{s,h}$. More precisely,
%Assumption~\ref{B123} below states assumptions on   The form of this assumption highlights that the coupling approach we use does not require that $b(x)$ is derivable from a potential.  We briefly comment on each condition appearing in Assumption~\ref{B123}.  The condition $b(0)=0$  is made to simplify the analysis.  Condition (B1) reflects that the measure $\mu$ is a small perturbation of the Gaussian reference measure $\mu_0$ when projected into directions associated with $e_j$ for $j$ large.  To precisely state (B1), 
fix $n \in \mathbb{N}$. Let $\H^{s,\ell} := \operatorname{span}\{e_1, \dots, e_n\}$, and let $\H^{s,h}$ denote its orthogonal complement, i.e., $\H^{s,h}$ is the closure in $\H^s$ of $ \operatorname{span}\{e_{n+1},e_{n+2}, \dots\} $.
Thus $\H^s = \H^{s,\ell} \oplus \H^{s,h}$.  For any $x \in \H^s$, we denote by $x^{\ell}$ and $x^h$ the orthogonal projections onto $\H^{s,\ell}$ and $\H^{s,h}$, respectively.

 \begin{assumption}\label{B123}
 $b$ is a function from $\H^s$ to $\H^s$ such that $b(0)=0$. Moreover it satisfies the following conditions:
 	\begin{itemize}
 		\item[(B1)] There exists $\L \in [1,\infty )$  such that 
 		\begin{equation}
  \snorm{b(x) - b(y)} \ \le\ \L \snorm{x - y} \quad \text{for all $x, y \in \H^{s}$.}  \label{eq:B1}		    	\end{equation}
 	\item[(B2)] There exists $n \in \mathbb{N}$ such that \begin{equation}
 	\snorm{b^h(x) - b^h(y) +   x^h  - y^h } \ \le\ \frac{1}{3} \snorm{x-y} \quad \text{for all $x, y \in \H^{s}$.}  \label{eq:B2}	 
 	\end{equation}
 	\item[(B3)]   There exist $K>0$ and $A \ge 0$ such that \begin{eqnarray} \label{eq:B3}
 	\langle x, b(x) \rangle_{s} &\le& - K \snorm{x}^2 + A \quad \text{for any $x \in \H^{s}$.}
 	\end{eqnarray}
  	\end{itemize}
 	 \end{assumption}
Condition (B1) is a global Lipschitz condition. Since $b(0)=0$, it implies the linear growth condition $\snorm{b(x)} \le \L \snorm{x}$, and this condition and (B3) imply that $K \le \L$. Condition (B2)
says that in the high modes, $b(x)$ behaves essentially as
a linear drift. Finally, Condition
(B3) is a standard drift condition which implies that the Markov chain has a Foster-Lyapunov function. It is similar
to other conditions in the literature that consider Markov processes on unbounded spaces based on second-order dynamical systems including Hypothesis~(H2) in \cite{BoSa2016}, Equation (13) of \cite{SaSt1999}, Hypothesis 1.1 in \cite{Ta2002}, Condition 3.1 in \cite{MaStHi2002}, and Assumption 1.2 in \cite{eberle2019couplings}.

\begin{lemma}[Foster-Lyapunov function] \label{lem:LYAP}
Suppose that Assumptions \ref{C123} and \ref{B123} hold. Then for any $T>0$ satisfying $
\L T^2 \le \frac{1}{48} \frac{K}{\L}$ {we have}
  \[
\E \left[ \snorm{X'(x)}^2 \right] \le \left( 1- \frac{K T^2}{2} \right) \snorm{x}^2 + 5 ( A + \tr(\tilde{\mathcal{C}} \mathcal{C}^{-s}) ) T^2 \quad \text{for all $x \in \H^s$.}
\]
\end{lemma}
%\mycomment{It seems possible to strengthen this result to an exponential Foster-Lyapunov function of the form $e^{\beta |x|^2}$ for $\beta$ sufficiently small.}
The proof of this lemma is given in Section \ref{sec:LYAPproof}.

\begin{example}[pHMC] \label{ex:U_Lg}
Suppose that $b$ is given by \eqref{eq:b} and $U$ satisfies the global Lipschitz condition in Assumption~\ref{MU123} with Lipschitz constant $L_g$. Then condition (B1) holds with Lipschitz constant $\L = 1 + \lambda_1^{1-s} L_g$ and Condition (B2) holds  with $n=\inf\{ k \in \mathbb{N} : \lambda_{k+1}^{1-s} < 1/(3 L_g) \}$.  Indeed, by
\eqref{eq:gradlip},
\begin{align*}
&\snorm{b(x) - b(y)} 
%\le \snorm{x-y} + \snorm{\mathcal{C} (D U(x) - D U(y) )}  \\
 \le \snorm{x-y} + \norm{\mathcal{C}^{1-s}  ( DU(x) - D U(y) )}_{-s}  \le ( 1 + \lambda_1^{1-s} L_g ) \snorm{x-y},
\end{align*}
and $\snorm{b^h(x) -  b^h(y) + x^h - y^h}\le \lambda_{n+1}^{1-s} L_g \snorm{x - y} \le (1/3) \snorm{x- y}$ as required. Moreover, the
drift condition (B3) can be verified in examples, see Section 
\ref{sec:applications}.
\end{example}

\begin{comment}
\begin{example}[pHMC with modified preconditioning] \label{ex:U_Lga}
Suppose that
$b(x)=-\tilde{\mathcal{C}} \mathcal{C}^{-1} x - \tilde{\mathcal{C}} DU(x)$, $U$ satisfies the global Lipschitz condition in Assumption~\ref{MU123} with Lipschitz constant $L_g$, and
$\tilde{\mathcal{C}}$ has eigenvalues \[
\tilde \lambda_i = \lambda_n \quad 1 \le i \le n \;, \quad \tilde \lambda_i = \lambda_i \quad i > n \;,
\] where $n=\inf\{ k \in \mathbb{N} : \lambda_{k+1}^{1-s} < 1/(3 L_g) \}$.  In this case, \begin{align*}
\snorm{b(x) - b(y)} 
%\le \snorm{x-y} + \snorm{\mathcal{C} (D U(x) - D U(y) )}  \\
 &\le \snorm{\tilde{\mathcal{C}} \mathcal{C}^{-1}( x-y) } + \norm{\tilde{\mathcal{C}} \mathcal{C}^{-s}  ( DU(x) - D U(y) )}_{-s}  \\
 &\le ( 1 + \lambda_1^{1-s} L_g ) \snorm{x-y},
\end{align*}
XXX 
\end{example}

\begin{example}[pHMC with non-gradient drift] \label{ex:U_Lgb}
XXX 
\end{example}
\end{comment}

\subsection{Two-scale coupling}\label{sec:main_results:couplings}
We now introduce a coupling for the transition steps of two copies of the Markov chain starting at different initial conditions $x$ and $y$. 
%The coupling is defined in a different way depending on whether $x$ and $y$ are far apart or sufficiently close.   
%\subsubsection{Synchronous coupling for $\anorm{x-y} \ge R$}\label{sec:main_results:couplings:synch_coupling} 
%The easiest way to couple the transition probabilities $\pi (x,\cdot)$
%and $\pi (y,\cdot)$ for two states $x,y\in\mathbb \H^s$ is to use the same random
%variable $\xi$  in both cases for the momentum refreshment. The corresponding coupling
%transition is given by $(x,y)\mapsto (X'(x,y),Y'(x,y))$ where
%\begin{equation}\label{**S}
%	X'(x,y) \ = \ q_T(x,\xi) \;, \quad 
%	Y'(x,y) \ = \ q_T(y,\xi) \;,
%\end{equation} 
%where $\xi \sim \mathcal{N}(0, \mathcal{C})$.  We apply synchronous coupling for $\anorm{x-y} \geq R$, which implies that $(x,y) \notin S$, and hence, in this region we can exploit the Foster-Lyapunov drift condition in \eqref{eq:FL_coupling} to ensure contractivity for the coupling.
%\subsubsection{A contractive coupling for $\anorm{x-y}< R$}
%For $\anorm{x-y}< R$ 
We use a synchronous coupling of the high modes in $\H^{s,h}$ and a different coupling for the low modes in $\H^{s,\ell}$ that together enable us to derive a weak form of contractivity.  
Note that the covariance operator $\tilde C$ has a bounded inverse 
on the finite dimensional subspace $\H^\ell$. Therefore,
for $\mathsf{h}\in \H^\ell$, the Gaussian measure $\mathcal{N}(\mathsf{h}, \tilde{\mathcal{C}})$ is absolutely continuous w.r.t.\ $\mathcal{N}(0, \tilde{\mathcal{C}})$ with relative density
\begin{equation}\label{rhoh}
  \rho_{\mathsf{h}}(x)\ =\ \exp  \left( \langle \tilde{\mathcal{C}}^{-1} \mathsf{h}, x \rangle -  \langle \tilde{\mathcal{C}}^{-1} \mathsf{h}, \mathsf{h} \rangle /2\right) .
\end{equation}
Let $\gamma>0$ be a positive constant. The precise value of the parameter 
$\gamma$ will be chosen in an appropriate way below.
The coupling transition step is given by $(x,y)\mapsto (X'(x,y),Y'(x,y))$ where
\begin{equation}\label{**}
	X'(x,y) \ = \ q_T(x,\xi) \;, \quad \text{and} \quad
	Y'(x,y) \ = \ q_T(y,\eta) 
\end{equation} 
with $\xi \sim \mathcal{N}(0, \tilde{\mathcal{C}})$ and $\eta$ defined in high/low components as $\eta^h \ :=  \ \xi^h$ and 
\begin{eqnarray}
	\label{eta}
	\eta^{\ell} &:=& \begin{cases}
		\xi^{\ell}  \ + \ \gamma  z^{\ell}  & \text{if } \ \mathcal U \ \leq \ 
		\rho_{-\gamma z^\ell}
%		\dfrac{d \mathcal{N} (-\gamma z^{\ell}, \tilde{\mathcal{C}})}{d \mathcal{N}(0, \tilde{\mathcal{C}})}
		(\xi^{\ell})
		 ,
		\\
		\mathcal{R} \xi^{\ell}  & \text{otherwise}.
		\end{cases}
\end{eqnarray}
Here $\mathcal U\sim\text{Unif}(0,1)$ is
independent of $\xi$, $z:=x-y$, and the reflection operator $\mathcal{R}$ is defined by
\begin{equation}
    \label{refl}
    \mathcal{R} \ :=\ \tilde{\mathcal{C}}^{1/2} (I - 2 e^{\ell} \langle e^{\ell}, \cdot \rangle ) \tilde{\mathcal{C}}^{-1/2},\quad\text{where}\quad e^{\ell}\ :=\ \tilde{\mathcal{C}}^{-1/2} z^{\ell}/\norm{\tilde{\mathcal{C}}^{-1/2} z^{\ell}}.
\end{equation}
%%%%%%%%%%%%%%%%%%%%%%%%%%%%%%%%%%%%%%%%%%%%%%%%%%%%%%%%%%%%%%%%%%%%%%%%
Due to Assumption~\ref{B123} (B2), the component in $\H^{s,h}$ of the resulting coupled dynamics is contracting in a finite time interval as a result of the linear part of the drift in \eqref{eq:dynamicsgeneral}.
Moreover, the coupling of the components of the initial velocities  
in $\H^{s,\ell}$ is similar to the coupling in \cite{BoEbZi2018} which is inspired by a related coupling for second order Langevin diffusions \cite{eberle2019couplings}. It is defined
in such a way that $\xi^{\ell}-\eta^{\ell}=-\gamma z^{\ell}$ occurs with
the maximal possible probability.
 As illustrated in Figure~\ref{fig:coupling}, and proven later in Lemma~\ref{lem:contr}, the reason for this choice is that the projection of the difference process on $\H^{s,\ell}$, i.e., $q_t^{\ell}(x,\xi)-q_t^{\ell}(y,\eta)$, is contracting in a finite time interval if the difference $\xi^{\ell}-\eta^{\ell}$ of the initial velocities is negatively proportional to the difference of the initial positions $x^{\ell}-y^{\ell}$. Note that
 if $b(x)= 0$ or $b(x)=-x$ then the optimal choices of $\gamma$ would be $\gamma=T^{-1}$ and $\gamma=\cot(T)$, respectively, because for these choices, $X'(x,y)=Y'(x,y)$ if $ \mathcal U \leq 
		\rho_{-\gamma z^\ell}(\xi^{\ell})$.
 In the case where $\xi^{\ell}-\eta^{\ell}\neq -\gamma z^{\ell}$, a 
 reflection coupling is applied. The corresponding reflection $\mathcal{R}$ is an orthogonal transformation w.r.t.\ the inner product
 $\langle x,y\rangle_{\tilde{\mathcal C}}=\langle \tilde{\mathcal{C}}^{-1/2}x,\tilde{\mathcal{C}}^{-1/2}y\rangle$
 induced by the covariance operator $\tilde{\mathcal{C}}$ on $\H^\ell$.
 \medskip

In order to verify that $(X'(x,y),Y'(x,y))$ is indeed a coupling of the transition probabilities $\pi (x,\cdot )$ and $\pi (y,\cdot ) $, we remark that the distribution of $\eta$ is 
$\mathcal{N}(0,\tilde{\mathcal{C}})$ since, by definition of $\eta^{\ell}$ in \eqref{eta} and a change of variables, \begin{eqnarray*}
\lefteqn{P[ \eta^{\ell} \in B ]}\\
&  =& E \left[ I_B( \xi^{\ell}+\gamma z^{\ell})\,  
%\dfrac{d \mathcal{N} (-\gamma z^{\ell}, \tilde{\mathcal{C}})}{d \mathcal{N}(0, \tilde{\mathcal{C}})}
\rho_{-\gamma z^\ell}
(\xi^{\ell}) \wedge 1 \right] \, +\, E \left[ I_B( \mathcal{{R}} \xi^{\ell} ) \left( 1 -   
%\dfrac{d \mathcal{N} (-\gamma z^{\ell}, \tilde{\mathcal{C}})}{d \mathcal{N}(0, \tilde{\mathcal{C}})}
\rho_{-\gamma z^\ell}
(\xi^{\ell} ) \right)^+ \right] \\
&  =& E \left[\rho_{\gamma z^\ell}
(\xi^{\ell})  I_B( \xi^{\ell})  
%\dfrac{d \mathcal{N} (-\gamma z^{\ell}, \tilde{\mathcal{C}})}{d \mathcal{N}(0, \tilde{\mathcal{C}})}
\rho_{-\gamma z^\ell}
(\xi^{\ell}-\gamma z^{\ell}) \wedge 1 \right] + E \left[ I_B(  \xi^{\ell} ) \left( 1 -   
%\dfrac{d \mathcal{N} (-\gamma z^{\ell}, \tilde{\mathcal{C}})}{d \mathcal{N}(0, \tilde{\mathcal{C}})}
\rho_{-\gamma z^\ell}
(\mathcal{{R}} \xi^{\ell} ) \right)^+ \right] \\
&=& E \left[ I_B( \xi^{\ell})\,  1\wedge 
%\dfrac{d \mathcal{N} (\gamma z^{\ell}, \tilde{\mathcal{C}})}{d \mathcal{N}(0, \tilde{\mathcal{C}})}
\rho_{\gamma z^\ell}
(\xi^{\ell})   \right] \, +\, E \left[ I_B( \xi^{\ell} ) \left( 1 -  
%\dfrac{d \mathcal{N} (\gamma z^{\ell}, \tilde{\mathcal{C}})}{d \mathcal{N}(0, \tilde{\mathcal{C}})}
\rho_{\gamma z^\ell}
(\xi^{\ell})   \right)^+ \right]  \ = \ P[\xi^{\ell} \in B] 
\end{eqnarray*}
for any measurable set $B$. Here $a \wedge b$ denotes the minimum of real numbers $a$ and $b$, $I_B(\cdot)$ denotes the indicator function for the set $B$, and we have used that
$\mathcal N(0,\tilde{\mathcal{C}})$ is invariant under the reflection $ \mathcal{R}$, $\mathcal{R} z^{\ell}= - z^{\ell}$, and 
by \eqref{rhoh},
$ \rho_{-\mathsf{h}}(x-\mathsf{h}) \rho_{\mathsf{h}}(x)=1$.  A similar calculation shows that
\begin{equation} \label{eq:TVnormal}
P[\eta^{\ell}\neq \xi^{\ell} +\gamma z^{\ell}]\ = \
E \left[  \left( 1 -   
%\dfrac{d \mathcal{N} (-\gamma z^{\ell}, \tilde{\mathcal{C}})}{d \mathcal{N}(0, \tilde{\mathcal{C}})}
\rho_{-\gamma z^\ell}
(\xi^{\ell} ) \right)^+ \right]
\ =\ d_{\TV}(\mathcal{N}(0,  \tilde{\mathcal{C}}) , \mathcal{N}(\gamma z^{\ell}, \tilde{\mathcal{C}}) )
\end{equation} where $d_{\TV}$ is the total variation distance.  Hence, by the coupling characterization of the total variation distance, $\eta^{\ell} = \xi^{\ell} + \gamma z^{\ell}$ does indeed hold with the maximal possible probability.  Note that if $z$ is not in the 
reproducing kernel Hilbert space of the covariance operator 
$ \tilde{\mathcal{C}}$ then the probability of the event $\eta^{\ell} \ne \xi^{\ell} + \gamma z^{\ell}$ in \eqref{eq:TVnormal} tends to one as the number of low modes increases. This explains why it is
necessary to split the Hilbert space and apply a two-scale coupling.    

%Therefore, $(X'(x,y),Y'(x,y))$ is indeed a coupling of the transition probabilities $\pi (x,\cdot )$ and $\pi (y,\cdot ) $. 

%%%%%%%%%%%%%%%%%%%%%%%%%%%%%%%%%%%%%%%%%%%%%%%%%%%%%%%%%%%%%%%%%%%%%%%%
\begin{figure}[t]
\centering
\begin{minipage}{\textwidth} 
      \begin{minipage}[b]{0.50\textwidth}
\begin{tikzpicture}[scale=1.0]
\begin{scope}[very thick] 
       \tikzmath{\gamm=0.3; \lamb=1/\gamm; 
       		     \x1 = 0.0; \x2 =1; \y1=1.0; \y2=-0.5;  \u1=\x1+1; \u2=\x2+0;
		     \z1=\x1-\y1; \z2=\x2-\y2;
                      \v1=\y1+(\u1-\x1)+\gamm*(\x1-\y1); 
                      \v2=\y2+(\u2-\x2)+\gamm*(\x2-\y2); 
                      \px1=\x1+\lamb*(\u1-\x1);   \px2=\x2+\lamb*(\u2-\x2); 
                      \py1=\y1+\lamb*(\v1-\y1); \py2=\y2+\lamb*(\v2-\y2); 
                      } 
                      \draw[->,gray,-{Latex[length=4mm]}](\y1,\y2)--(\x1,\x2) node [midway, below, black, scale=1.5]  {$z_i$};
\filldraw[color=black,fill=black] (\x1,\x2) circle (0.1) node [left,black, scale=1.5]  {$x_i$};
\filldraw[color=black,fill=black] (\y1,\y2) circle (0.1) node [below,black, scale=1.5]  {$y_i$};
\draw[->,-{Latex[length=2mm]}](\x1,\x2)--(\u1,\u2) node [above,black, scale=1.5] {$\xi_i$};
\draw[->,-{Latex[length=2mm]}](\y1,\y2)--(\v1,\v2) node [right,black, scale=1.5]  {$\eta_i = \xi_i+\gamma z_i$};
\draw[-,dotted](\x1,\x2)--(\px1,\px2);
\draw[-,dotted](\y1,\y2)--(\py1,\py2);
\filldraw[color=black,fill=black] (\py1,\py2) circle (0.1);
 \end{scope}
\end{tikzpicture}
\end{minipage}
      \begin{minipage}[b]{0.50\textwidth}
\begin{tikzpicture}[scale=1.0]
\begin{scope}[very thick] 
       \tikzmath{\gamm=0.0;  
       		     \x1 = 0.0; \x2 =1; \y1=1; \y2=-0.5;  \u1=\x1+1; \u2=\x2+0;
		     \z1=\x1-\y1; \z2=\x2-\y2;
                      \v1=\y1+(\u1-\x1)+\gamm*(\x1-\y1); 
                      \v2=\y2+(\u2-\x2)+\gamm*(\x2-\y2); 
                      } 
\filldraw[color=black,fill=black] (\x1,\x2) circle (0.1) node [left,black, scale=1.5]  {$x_i$};
\filldraw[color=black,fill=black] (\y1,\y2) circle (0.1) node [below,black, scale=1.5]  {$y_i$};
\draw[->,-{Latex[length=2mm]}](\x1,\x2)--(\u1,\u2) node [above,black, scale=1.5] {$\xi_i$};
\draw[->,-{Latex[length=2mm]}](\y1,\y2)--(\v1,\v2) node [right,black, scale=1.5]  {$\eta_i = \xi_i$};
 \end{scope}
\end{tikzpicture}
\end{minipage}
\end{minipage}
\begin{minipage}{\textwidth}
      \begin{minipage}[b]{0.48\textwidth}
      \centering
(a) low modes: $1 \le i \le n$
\end{minipage}
\begin{minipage}[b]{0.48\textwidth}
\centering
(b) high modes: $i >n$
\end{minipage}
\end{minipage}
\caption{\small {\bf Two-Scale Coupling.} A diagram illustrating the two-scale coupling in the case $\gamma=T^{-1}$. 
(a) In the low modes $(1 \le i \le n)$, the initial velocities are coupled such that the final positions $q_T^{\ell}(x,\xi) = q_T^{\ell}(y,  \eta)$ are equal when $b \equiv 0$.  (b)  In the high modes $(i>n)$, the initial velocities are synchronously coupled. }
  \label{fig:coupling}
\end{figure}
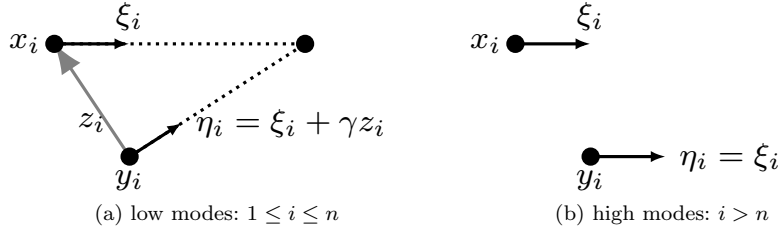
%%%%%%%%%%%%%%%%%%%%%%%%%%%%%%%%%%%%%%%%%%%%%%%%%%%%%%%%%%%%%%%%%%%%%%%%

\subsection{Contractivity}  \label{sec:main_results:contraction}
We now state our main contraction bound for the coupling introduced above.
We first define a norm $\anorm{\cdot}$ on $\H^s$ where the high modes are weighted by $\alpha>0$: \begin{equation} \label{eq:alphanorm}
\anorm{x}\ =\  \norm{\tilde{\mathcal{C}}^{-\frac{1}{2}}  x^{\ell}} + \alpha \snorm{x^h}
\ =\  \snorm{\tilde{\mathcal{C}}^{-\frac{1}{2}} \mathcal{C}^{\frac{s}{2}} x^{\ell}} + \alpha \snorm{x^h} \;. 
\end{equation} Let $\sigma_{min} = \min_{1 \le i \le n} \{ \tilde{\lambda}_i^{-1/2} \lambda_i^{s/2}\}$  and 
$\sigma_{max} = \max_{1 \le i \le n} \{ \tilde{\lambda}_i^{-1/2} \lambda_i^{s/2} \}$.
Note that 
\begin{equation}
    \label{compnorm}
    \sigma_{min} \snorm{x^\ell}\ \le\ \snorm{\tilde {\mathcal{C}}^{-1/2} \mathcal{C}^{s/2}x^\ell  }  \ \le\  \sigma_{max}  \snorm{x^\ell} \quad \text{for all $x \in \H^{s}$}.
\end{equation}
Thus
$\anorm{\cdot}$ and $\snorm{ \cdot }$ are equivalent norms  with
\begin{equation} \label{compnormA}
 \min(\sigma_{min}, \alpha) \snorm{x} \le \anorm{x} \le  \sqrt{2} \max(\sigma_{max}, \alpha)  \snorm{x} \quad \text{for all $x \in \H^{s}$}.
\end{equation}

\begin{remark}
If the dimension is infinite then
the operator $\tilde{\mathcal{C}}^{-1} \mathcal{C}^{s}$ is unbounded on $\H^s$, because its inverse is trace class.  Nonetheless, $\anorm{ x }$ is a well-defined norm for any $x \in \H^s$ because the operator $\tilde{\mathcal{C}}^{-\frac{1}{2}} \mathcal{C}^{\frac{s}{2}}$ appearing in $\anorm{ x }$ only acts on the
projection $x^{\ell}$ of $x$ onto the finite dimensional space $\H^{s,\ell}$.
\end{remark}

%\subsection{Foster-Lyapunov function} A key ingredient to proving our main result is a quadratic Foster-Lyapunov function for pHMC, which we develop in the following theorem.

%Next we specify a coupling for pHMC that is contractive within $S$.

%For given $x,y\in\H$ let $r(x,y)\ =\ |x-y|,\qquad R'(x,y)\ =\ |X'(x,y)-Y'(x,y)|$$ denote the coupling distance before and after the coupling transition step.  
As we will see below, even when $U$ is non-convex, we can still obtain contractivity with respect to a semimetric $\rho: \H^s \times \H^s \to [0, \infty)$ of the form
\begin{equation}\label{rho}
\rho(x,y) \ =\ \sqrt{ f( \anorm{x-y}) (1+ \epsilon \snorm{x}^2 + \epsilon \snorm{y}^2 ) },\qquad x,y\in\H^s,
\end{equation} 
where $f:[0,\infty )\to [0,\infty )$ is a concave function given by \begin{equation} \label{eq:f}
f(r) = \int_0^r e^{-a t} \ I_{\{ t \le R\}} \ dt \ =\ \frac 1a\left( 1-e^{-a\, r\wedge R}\right) \;,
\end{equation}
and where $R>0$,  $a >0$, and $\epsilon > 0$ are parameters to be specified below.  The semimetric $\rho$ is similar to the one introduced in \cite{BoEbZi2018} in order to prove contractivity 
of the HMC transition step in the finite dimensional case. In general, $\rho$ is not a metric, since the triangle inequality might be violated.   Note that $f$ is non-decreasing, and constant when $r \ge R$.

\begin{remark}
The semimetric \eqref{rho} incorporates, in a multiplicative or weighted way, the quadratic Foster-Lyapunov function for pHMC from Lemma~\ref{lem:LYAP} with weight $\epsilon$.  The idea to use semimetrics of this general form to study contraction properties of Markov processes goes back to \cite{HairerMattinglyScheutzow} and \cite{Butkovsky}; see also \cite{eberle2019quantitative}.
\end{remark}

Lemma~\ref{lem:LYAP} implies that 
%any coupling of two copies of the pHMC chain with 
the coupling transition $(x,y) \mapsto (X'(x,y), Y'(x,y))$ also has a quadratic Foster-Lyapunov function: If $
\L T^2 \le  \frac{1}{48} \frac{K}{\L}$ then
\[
\E \left[ \snorm{X'(x,y)}^2 + \snorm{Y'(x,y)}^2 \right] \le \left( 1- \frac{K T^2}{2} \right) ( \snorm{x}^2 + \snorm{y}^2 ) + 10 ( A + \tr(\tilde{\mathcal{C}} \mathcal{C}^{-s}) ) T^2 \;.
\]  
We fix a finite, positive constant $R$ satisfying
\begin{equation} \label{ieq:R}
R \ \ge\ 8 \sqrt{40}   ( A + \tr(\tilde{\mathcal{C}} \mathcal{C}^{-s}))^{1/2}\sigma_{max} LK^{-1/2} \;.
\end{equation}
In our main result below, we choose $\alpha :=4\sigma_{max}L$.
In this case, the choice of $R$ in \eqref{ieq:R} guarantees that a strict drift condition \begin{equation} \label{eq:FL_coupling}
\E \left[ \snorm{X'(x,y)}^2 + \snorm{Y'(x,y)}^2 \right]\ \le\ \left( 1- {K T^2}/{4} \right) ( \snorm{x}^2 + \snorm{y}^2 ) 
\end{equation} holds for all $(x,y)$ satisfying $\anorm{x - y}\ge R $, because 
by \eqref{compnormA} and since $L\ge 1$,
\begin{equation}\label{compnormB}
  \anorm{x - y}\ \le\ 4\sqrt 2\,  \sigma_{max}L\snorm{x-y}\ \le\ 8 \sigma_{max}L\sqrt{ \snorm{x}^2 + \snorm{y}^2}  .
 \end{equation}
The asymptotic strict drift condition in \eqref{eq:FL_coupling}  allows us to split the proof of contractivity into two parts: (i) $\anorm{x - y}\ge R $ where any coupling is contracting in $\rho$ due to \eqref{eq:FL_coupling}, and (ii) $\anorm{x - y}< R $, where  $\rho$ is contracting due to the specially designed two-scale coupling.

\begin{theorem} \label{thm:CONTR}
 Suppose that Assumption  \ref{B123} holds. Let $T>0$ satisfy
\begin{equation}
\label{A0A}
\dfrac{\sigma_{max}}{\sigma_{min}} \L T^2\ \le\ \min \left( \dfrac{1}{48} \dfrac{K}{\L}, \dfrac{1}{256 \L R^2} \dfrac{\sigma_{min}}{\sigma_{max}} \right)
\end{equation}
Let $\alpha$, $\gamma$, $a$, and $\epsilon$ be given by \begin{eqnarray}
\label{Calpha}\alpha &:=& 4 \sigma_{max} \L \;,\\
\label{Cgamma}\gamma &:=& \min \left( T^{-1},R^{-1}/4\right) \;,\\
\label{Ca}a &:=&  T^{-1} \;, \\
\label{Ce} \epsilon &:= &  (1/160) (A+\tr(\tilde{\mathcal{C}} \mathcal{C}^{-s}))^{-1} e^{-  R / T } \;.
\end{eqnarray} 
Then for any $x, y \in \H^s$, we have
\begin{eqnarray}\label{contrmain}
	E[\rho(X'(x,y),Y'(x,y))] &\leq& e^{-c} \rho(x,y),\qquad\mbox{where}
\end{eqnarray}
\begin{equation}\label{crate}
c\ =\  \min\left(\frac{1}{16} K T^2,\,  \frac{1}{128} T\max (R,T)\,  e^{-\max (R,T)/T} \right) \;.
\end{equation}
\end{theorem}

\medskip

The proof of this theorem is given in \S\ref{sec:CONTRproof}.

\begin{remark}
The rate in \eqref{crate} is similar to the rate in the finite-dimensional case found in Theorem 2.3 of Ref.~\cite{BoEbZi2018}. The main difference is that the condition on $\L T^2$ in \eqref{A0A} now reflects the effect of preconditioning.  
\end{remark}

\subsection{Quantitative bounds for distance to the invariant measure}
\label{sec:QB}

Theorem \ref{thm:CONTR} establishes global contractivity of the transition kernel $\pi (x,dy)$ w.r.t. the Kantorovich distance based on the underlying semimetric $\rho$, which  for probability measures $\nu ,\eta$ on $\H^s$ is defined as
$$\mathcal W_\rho (\nu ,\eta )\ =\  \inf_{\gamma\in C(\nu ,\eta )}\int\rho (x,y)\,\gamma (dx\, dy) = 
\inf_{\substack{X' \sim \nu,  Y' \sim \eta}}E \left[ \rho(X',Y') \right]
$$
where the infimum is over all couplings $\gamma$ of $\nu$ and $\eta$. 
Moreover, it implies quantitative bounds for the standard $L^1$
Wasserstein distance 
$$\mathcal W^{s,1} (\nu ,\mu )\ =\  \inf_{\gamma\in C(\nu ,\mu )}\int \snorm{x-y}\,\gamma (dx\, dy) = 
\inf_{\substack{X' \sim \nu,  Y' \sim \mu}}E \left[ \snorm{X'-Y'} \right]
$$
with respect to the invariant measure $\mu$ on $\H^s$. Let $M_1(\nu ):=
\int \snorm{x}\, \nu (dx)$.
%$\mathcal W_\rho$ immediately implies a quantitative bound on the standard
%$L^1$-Wasserstein distance
%$$\mathcal W^1 (\nu\pi^n ,\mu )\ =\  \inf_{\gamma\in C(\nu\pi^n ,\mu )}\int |x-y|\,\gamma (dx\, dy)$$
%between the law of the pHMC chain after $n$ steps and the invariant probability measure $\mu$ in \eqref{eq:mu}.

\begin{corollary}\label{cor:QBHMC}
Suppose that Assumption  \ref{B123} holds.  Let $T\in (0,R)$ satisfy \eqref{A0A}.  Then for any $k\in\mathbb N$ and for any probability measures $\nu ,\eta$ on $\H^s$,
\begin{eqnarray}
\label{QBpHMC1}
\mathcal W_\rho (\nu\pi^k,\eta\pi^k)& \le & e^{-c k}\, \mathcal W_\rho (\nu ,\eta ),\qquad\qquad \text{and}\\
\mathcal W^{s,1} (\nu\pi^k ,\mu )& \le & C\,(1+\sqrt\epsilon M_1(\nu )+ (1/4) K^{-1/2} e^{-R/(2 T)})\, e^{-c k} \label{QBpHMC3}
\end{eqnarray}
where the rate $c$ and the constant $\epsilon$ are given explicitly by \eqref{crate} and \eqref{Ce}, 
and
\begin{equation}\label{Cdefi}
  C\ =\ \max \left( 2 {T}\sigma_{min}^{-1},\,23\, ({A+\tr(\tilde{\mathcal{C}} \mathcal{C}^{-s})})^{1/2} e^{  R /(2 T) }\right).   
\end{equation}
In particular, for a given 
constant $\delta\in (0,\infty )$, the $L^1$ Wasserstein distance $\Delta (k)=\mathcal W^{s,1}(\nu\pi^k ,\mu )$ w.r.t.\ $\mu$ after $k$ steps of the chain with initial distribution $\nu$ satisfies $\Delta (k)\le\delta$
provided
\begin{equation}
\label{QBpHMC2}
k\ \ge\ \frac 1c\, \log\frac{C \,(1+ \sqrt{\epsilon} M_1(\nu )+(1/4) K^{-1/2} e^{-R/(2 T)})}{\delta} . 
%k\ \ge\ \frac 1c\, \log\frac{C \,(1+ \sqrt{\epsilon} M_1(\nu )+\sqrt{\epsilon} M_1(\mu ))}{\delta} . 
\end{equation}
\end{corollary}

The corollary is a rather direct consequence of Theorem \ref{thm:CONTR}.
A short proof is included in \S\ref{sec:CONTRproof}.

\begin{remark}[Quantitative bounds for ergodic averages] 
MCMC methods are often applied to approximate expectation values w.r.t.\ the target distribution by ergodic averages of the Markov chain. Our results (e.g.\ \eqref{QBpHMC1}) directly imply completely explicit bounds for
bias and variances, as well as explicit concentration inequalities for these ergodic averages in the case of pHMC. Indeed, the general results by Joulin and Ollivier \cite{JoulinOllivier} show that such bounds follow directly from an $L^1$ Wasserstein contraction w.r.t.\ an arbitrary metric $\rho$, which is precisely the statement shown above.
 \end{remark}

\section{Applications}
\label{sec:applications}

%We apply our main results to the following concrete examples. 
%where preconditioned Hamiltonian Monte Carlo is used to sample from probability distributions on a Hilbert space $\H$ that have the general form \begin{equation} \label{eq:mu}
%\mu(dx) \propto \exp(- U(x) ) \mu_0(dx) \;, \quad \mu_0 = \mathcal{N}(0, \mathcal{C}) \;,
%\end{equation}
%where $U: \H \to \mathbb{R}$ is a potential energy and $\mathcal{N}(0, \mathcal{C})$ denotes a centered Gaussian reference measure with covariance operator $\mathcal{C}$. 
\subsection{Transition Path Sampling} \label{sec:TPS} 
Here we discuss the use of pHMC 
in transition path sampling (TPS). As an application of Theorem~\ref{thm:CONTR}, we obtain dimension-free contraction rates for exact preconditioned HMC in this context.  Fix a time horizon $\tau>0$ (not to be confused with the duration parameter in preconditioned HMC which we denote by $T$).   The aim of TPS \cite{ReVa2005, hairer2007analysis, BeRoStVo2008, HaStVo2009} is to sample from a \textit{diffusion bridge} or \textit{conditioned diffusion}, i.e., from the conditional law $\nu_{a,b}$ of the solution $\mathsf{X}: [0,\tau] \to \mathbb{R}^d$ to a $d$-dimensional stochastic differential equation of the form \begin{equation} \label{sde}
\mathsf{d} \mathsf{X}( \mathsf{t} ) = - \nabla \Psi( \mathsf{X}( \mathsf{t} ) ) \, \mathsf{ds} + \mathsf{d} \mathsf{W}( \mathsf{t} )
\end{equation} given both initial and final conditions \[
\mathsf{X}(0) = a \quad \text{and} \quad \mathsf{X}(\tau) = b .
\] 
Here $\Psi: \mathbb{R}^d \to \mathbb{R}$ is a given potential energy function and $ \mathsf{W}$ is a $d$-dimensional 
standard Brownian motion. 
%The process $\mathsf{X}$ is known as a \textit{diffusion bridge} or \textit{conditioned diffusion} \cite{ReVa2005, hairer2007analysis, BeRoStVo2008, HaStVo2009}.
TPS is particularly relevant to molecular dynamics where the states $a$ and $b$ represent different configurations of a molecular system \cite{Bolhuis2002,Miller2007,pinski2010transition}.\smallskip

We first recenter: Let $\mu=\nu\circ \theta_M^{-1 }$ denote the law 
of the recentered bridge where $\theta_M(x)=x-M$ is the translation 
on path space by the mean 
$\mathsf{M}(\mathsf{t}) =  a + (\mathsf{t}/\tau) (b-a)$ of the 
Brownian bridge from $a$ to $b$.
Then by Girsanov's theorem, the measure $\mu$ is absolutely continuous 
%with $\Psi\equiv0$ in \eqref{sde}, and
%The diffusion bridge $\mathsf{X}-\mathsf{M}$ has a path density 
with respect to the law $\mu_0$ of the Brownian bridge from $0$ to $0$ \cite{hairer2005analysis,BeRoStVo2008}.
Moreover, the measure $\mu_0$ is
the centered Gaussian measure on the Hilbert space $\H = L^2([0,\tau], \mathbb{R}^d)$ with covariance operator $\mathcal{C}=-\Delta_D^{-1}$ where $\Delta_D$ is the Dirichlet Laplacian, and the relative density of $\mu$
%The operator $\mathcal{C}$ has orthonormal eigenfunctions and corresponding eigenvalues given respectively by \begin{equation} \label{eq:dirichlet_eigs} e_{k+j}(\mathsf{t}) = \sqrt{\frac{2}{\tau}} \sin\left(\frac{k \pi \mathsf{t}}{\tau} \right) \vec{e}_{j+1} \;, \quad \lambda_{k+j}=\frac{\tau^2}{k^2 \pi^2},  \quad \mathsf{t} \in [0,\tau] \;,  \end{equation}  where  $0 \le j \le d-1$, $k \in \mathbb{N}$, and $\{ \vec{e}_i \}_{1 \le i \le d}$ is the standard basis on $\mathbb{R}^d$.  
%It follows from \eqref{eq:dirichlet_eigs} that $\tr(\mathcal{C}) = d \tau^2 / 6$.   
%Moreover, the path density of the process $\mathsf{X}-\mathsf{M}$ 
with respect to $\mu_0$ is proportional to $\exp(-U(x))$ where the function $U(x)$ is defined in terms of the so-called path potential energy function $G : \mathbb{R}^d \to \mathbb{R}$ as follows \begin{equation} \label{eq:U_TPS}
U(x) = \int_0^{\tau} G ( x(\mathsf{t})+\mathsf{M}(\mathsf{t}) ) \mathsf{dt}  \quad \text{where} \quad  G(\cdot) = \frac{1}{2} | \nabla \Psi(\cdot)|^2 -   \frac{1}{2} \Delta \Psi(\cdot) \;.
\end{equation} 
In the main convergence result given below, we make the following regularity assumption on $G$.

\begin{assumption} \label{G123} The function $G: \mathbb{R}^d \to \mathbb{R}$ is continuously differentiable. Moreover, $\nabla G(0)=0$, and $\nabla G$ is uniformly bounded and globally
Lipschitz continuous, i.e., there exist finite constants $M_G, L_G$ such that for all $x,y\in\mathbb R^d$,
$$  | \nabla G(x)| \le M_G, \quad\text{and}\quad |\nabla G(x)-\nabla G(y)|\le L_G|x-y|. $$
%satisfies: $\nabla G(0)=0$; a global Lipschitz condition with Lipschitz constant $L_G$; and, $\nabla G$ is bounded, i.e., $\sup | \nabla G| \le M_G$ for some constant $M_G$.  
\end{assumption}
This regularity assumption frequently holds in molecular dynamics applications, since the configuration space of molecular systems is usually taken to be a fixed cubic box with periodic boundary conditions \cite{AlTi1987,FrSm2002,St2007,CaLeSt2007, ELi2008,Bo2014,Leimkuhler2015}. 
In this case, we can lift the TPS problem to the covering space $\R^d$
by extending the path potential to a periodic function on this space.
Thus after recentering the coordinate system, Assumption \ref{G123} is satisfied whenever $G$ is $C^2$.\smallskip

To implement TPS on a computer, we use the finite difference method to approximate the infinite-dimensional distribution $\mu(dx) \propto \exp(- U(x) ) \mathcal{N}(0, \mathcal{C})(dx)$ by a finite-dimensional probability measure $\mu_{m}$. Other approximations, e.g., Galerkin or finite-element, are also possible and should yield similar results.  We focus on the finite difference method because it is widely used in practice.  Discretize the interval $[0,\tau]$ into $m +2$ evenly spaced grid points \begin{equation} \label{grid_TPS}
 \mathsf{t}_j = {\tau}j/{(m+1)} ,  \quad j=0, \dots, m+1  \;.
\end{equation}   The space of paths on $\mathbb{R}^d$ is then approximated by the finite-dimensional space $\mathbb{R}^{m d}$.  Specifically, we write $\bph{x} \in \mathbb{R}^{m d}$ as \[
\bph{x} = (\bph{x}_{1:d}, \dots, \bph{x}_{m+1:m+d})
\] where the $j$-th component $\bph{x}_{j+1:j+d}:=(\bph{x}_{j+1}, \dots, \bph{x}_{j+d})$ is a $d$-dimensional  vector that can be viewed as an approximation of $x(\mathsf{t}_j)$ for $j=1,\dots,m$.  
The Dirichlet Laplacian $\Delta_D$ is approximated by the $m d \times m d$ Dirichlet Laplacian matrix $\bph{\Delta}_{D,m}$ with $(i,j)$-th entry \[
(\bph{\Delta}_{D,m})_{i,j} = \begin{dcases} - 2 \left( { \tau }/{(m+1)} \right)^{-2} & \text{if $|i-j|=0$,} \\
 \left( { \tau }/{(m+1)} \right)^{-2} & \text{if $|i-j|=d$,} \\
 0 & \text{otherwise.} 
\end{dcases} \] The covariance operator $\mathcal{C}$ is approximated by the $m d \times m d$ matrix $\bph{\mathcal{C}} = - \bph{\Delta}_{D,m}^{-1}$, and the Hilbert space $\mathcal{H}$ is represented by $\mathbb{R}^{m d}$ with inner product given by the weighted dot product $\langle \bph{x}, \bph{y} \rangle =  \frac{\tau}{m+1} \bph{x} \bullet \bph{y}$.  
The functional \eqref{eq:U_TPS} is discretized as \[
U_{m}( \bph{x}) = \frac{\tau}{m+1} G_{m}(\bph{x}) \;, \quad \text{where} \quad G_{m}(\bph{x}) = \sum_{j=1}^{m} G(\bph{x}_{j+1:j+d} + \mathsf{M}(\mathsf{t}_j) ) \;.
\] Note that if the vector $\bph{x}$ contains the grid values of a smooth function $x$, then $U_{m}( \bph{x}) \to U(x)$ as $m \to \infty$.  
In summary, the infinite-dimensional path distribution $\mu(dx)$ is approximated by the finite-dimensional probability measure $\mu_m(d \bph{x})$ with non-normalized density $\exp\left(-U_{m}( \bph{x})- \frac{1}{2} \langle \bph{x}, \bph{\mathcal{C}}^{-1} \bph{x} \rangle \right)$.

To approximately sample from $\mu_m$, we use pHMC with transition step in \eqref{eq:transitionstep}. This corresponds to a Markov chain on $\mathbb{R}^{m d}$ with transition step \begin{equation} \label{eq:transitionstep_TPS}
\bph{x} \mapsto \bph{X}'(\bph{x}) ~:=~ \bph{q}_T(\bph{x}, \bph{\xi}) \;, \qquad \bph{\xi} \sim \mathcal{N}(0, \frac{m+1}{\tau} \bph{\mathcal{C}})
\end{equation} where $\bph{q}_t$ solves  \begin{equation} \label{eq:dynamics_TPS}
\frac{d}{dt} \bph{q}_t = \bph{v}_t \;, \quad \frac{d}{dt} \bph{v}_t = \bph{b}(\bph{q}_t)  \;, \quad (\bph{q}_0( \bph{x},\bph{v} ), \bph{v}_0( \bph{x},\bph{v} )) = ( \bph{x},\bph{v} ) \in \mathbb{R}^{2 m d} \;,
\end{equation}
with $\bph{b}(\bph{x}) = -  \bph{x} - \bph{\mathcal{C}} \nabla  G_{m} (\bph{x})$.  Let $\pi_m$ denote the transition kernal of \eqref{eq:transitionstep_TPS}.

\begin{theorem}[Transition Path Sampling]
\label{thm:CONTR_TPS}
Suppose that Assumption~\ref{G123} holds.
 Let $\kappa :=2 (\tau^2 / \pi^2) L_G$, 
 $m_{\ell}=\lfloor \sqrt{3 \kappa} \rfloor$,
 $n=m_{\ell} d$, and 
$m^{\star}= \lceil  (m_{\ell}+1) \pi/2 \rceil$. Let $R$, $c$, $C$ and $\epsilon$ be defined as
\begin{align}
R \ &:= \ 16 \sqrt{20} \pi \kappa^{1/2} (1+\kappa) ( (\tau / \pi)^3 M_G^2 +  d)^{1/2} \;, \label{R_TPS} \\
 \quad c \ &:= \  \min( (1/32) T^2, (1/128) T^2 \max(R,T) e^{-\max(R,T)/T} ) \;, \label{crate_TPS} \\
C \ &:= \ \max(T \tau, 23 ( (\tau^5 / \pi^4) M_G^2+ d \tau^2 / 3  )^{1/2} e^{R/(2 T)})  \;,  \label{C_TPS} \\
\epsilon \ &:= \  (  \tau^5 M_G^2)^{-1} e^{-R/T} \;. \label{eps_TPS}  
\end{align}
Suppose that the duration parameter $T \in (0, R)$ satisfies
\begin{align}
 2 \sqrt{3 \kappa} (1+\kappa) T^2 \ &\le \  \min\left( \frac{1}{96  (1+\kappa)}, \frac{1}{512 \sqrt{3 \kappa} (1+\kappa) R^2} \right) \;.  \label{A0A_TPS} 
\end{align}
Then for any $m > m^{\star}$, $k\in \mathbb{N}$, and probability measure $\nu_m$ on  $\mathbb{R}^{m d}$,
\begin{align}
\mathcal W^{0,1}  (\nu_m \pi_m^k, &\mu_m) \le  C e^{-c k} \left( 1 + \sqrt{\epsilon} M_1(\nu_m) + (1/8) e^{-R/(2 T)} \right) \;. \label{QBpHMC_TPS}
\end{align}
\end{theorem}

\begin{remark}
 Note that the upper bound in \eqref{QBpHMC_TPS} depends on dimension only through the initial distribution.   The dimension independence in the other terms of this bound reflects that the finite-dimensional pHMC algorithm in \eqref{eq:transitionstep_TPS} converges to an infinite-dimensional pHMC algorithm whose transition kernel satisfies an infinite-dimensional analog of this quantitative bound.
\end{remark}

A proof of this result is given in \S\ref{sec:TPS_proofs}.

\subsection{Path Integral Molecular Dynamics}  \label{sec:PIMD}
Here we discuss the use of pHMC for path-integral molecular dynamics (PIMD), and as an application of Theorem~\ref{thm:CONTR}, obtain dimension-free contraction rates for preconditioned HMC in this context.  PIMD is used to compute exact Boltzmann properties and approximate dynamical properties of quantum mechanical systems \cite{ChandlerWolynes}. The technique is based on Feynman's path-integral formulation of quantum statistical mechanics \cite{Feynman1965}, and the observation that the quantum Boltzmann statistical mechanics of a quantum system can be reproduced by the {\em classical} Boltzmann statistical mechanics of a ring-polymer system \cite{ChandlerWolynes}.

Consider $N$ interacting quantum particles in 3D with potential energy operator given by
\begin{equation}
	\hat{V}=V(\hat{q}_1, \dots, \hat{q}_N)
	\label{eq:H_q}
\end{equation}
where $\hat{q}_i$ is the three-dimensional 
position operator of particle $i$ and $V: \mathbb{R}^d \to \mathbb{R}$ is a potential energy function where $d=3 N$ \cite{Habershon2013}.  The thermal equilibrium properties of this system are described by the quantum mechanical Boltzmann partition function,
\begin{equation}
	Q=\operatorname{Tr}[e^{-\beta \hat{V}}] 
\end{equation}
where $\beta$ is an inverse temperature parameter.  For some $\mathsf{a}>0$, suppose that the potential energy function can be written as \[
V(\cdot) = \frac{1}{2} \mathsf{a} |\cdot|^2 + G(\cdot)
\] where $G: \mathbb{R}^d \to \mathbb{R}$.  Then the partition function $Q$ can be written as the expected value of a Gaussian random variable on loop space as follows
\begin{equation}
    Q=\E[ e^{-U(\xi)} ] \;, \quad \text{where} \quad \xi \sim \mu_0 = \mathcal{N}(0,\mathcal{C}_{\mathsf{a}}) \;,
    	\label{eq:path_integral}
\end{equation}  and the covariance operator $\mathcal{C}_{\mathsf{a}}$ of the Gaussian reference measure is defined in terms of the  Laplacian with periodic boundary conditions $\Delta_P$ on $L^2([0,\beta],\mathbb{R}^{d})$ as follows \[
\mathcal{C}_{\mathsf{a}} = (-\Delta_P + \mathsf{a} I)^{-1} \;,
\] where $I$ is the identity operator and the potential energy $U(x)$ is given by \begin{equation}
U(x) = \int_0^{\beta} G(x(\mathsf{t}))  \mathsf{dt} \;.
    	\label{eq:U_PIMD}
\end{equation}
The probability measures $\mu_0$ and $\mu(dx) \propto \exp(- U(x) )\, \mathcal{N}(0, \mathcal{C}_{\mathsf{a}})(dx)$
are supported on the loop space consisting of all 
periodic continuous paths $x:[0,\beta ]\to\mathbb R^d$. 
They are similar to the corresponding measures 
considered for Transition Path Sampling, but 
there is an additional, artificially introduced
parameter $\mathsf{a}$ appearing in $\mathcal{C}_{\mathsf{a}}$.
This parameter is essential because $\Delta_P$ is not invertible since it has a zero (leading) eigenvalue corresponding to the `centroid mode' \cite{Lu2018}.\smallskip  

To implement PIMD on a computer, we use  finite-differences to truncate the infinite-dimensional path distribution $\mu$ to a finite-dimensional one $\mu_{m}$  by discretizing the interval $[0,\beta]$ into $m +1$ grid points \begin{equation} \label{grid_PIMD}
 \mathsf{t}_j = {\beta}j/{m} ,  \quad  j=0, \dots, m  \;.
\end{equation} The space of loops on $\mathbb{R}^d$ is approximated by the finite-dimensional space $\mathbb{R}^{m d}$.  Specifically, we write $\bph{x} \in \mathbb{R}^{m d}$ as \[
\bph{x} = (\bph{x}_{1:d}, \dots, \bph{x}_{m+1:m+d})
\] where  $\bph{x}_{j+1:j+d}:=(\bph{x}_{j+1}, \dots, \bph{x}_{j+d})$ is a $d$-dimensional vector that can be viewed as an approximation of $x(\mathsf{t}_j)$ for $j=1,\dots,m$.

\begin{remark}
Comparing \eqref{grid_TPS} to \eqref{grid_PIMD}, note that the number of grid points in TPS, resp.~PIMD, is $m+2$, resp.~$m+1$.  Nonetheless, in both cases path and loop space are approximated by $\mathbb{R}^{m d}$.  The difference in the number of grid points is due to the boundary conditions: in TPS the Dirichlet boundary conditions eliminate two unknown $d$-dimensional vectors, whereas in PIMD the periodic boundary conditions eliminate only one unknown $d$-dimensional vector.  Thus, the total number of unknowns in both cases is $m d$.
\end{remark}

The periodic Laplacian $\Delta_P$ is approximated by the $m d \times m d$ discrete periodic Laplacian matrix $\bph{\Delta_{P,m}}$ with $(i,j)$-th entry \[
(\bph{\Delta_{P,m}})_{i,j} = \begin{dcases} - 2 \left( \beta /{m} \right)^{-2} & \text{if $|i-j|=0$,} \\
 \left( \beta /{m} \right)^{-2} & \text{if~$(i-j) \bmod ( m d )=d~\text{or}~(j-i) \bmod (m d)=d$,} \\
 0 & \text{otherwise.} 
\end{dcases} \]  Naturally, the covariance operator $\mathcal{C}_{\mathsf{a}}$ is approximated by the $md \times md$ matrix $\bph{\mathcal{C}}_{\mathsf{a}} = (- \bph{\Delta_{P,m}} + \mathsf{a} \bph{I}_{m d \times m d})^{-1}$ where $\bph{I}_{m d \times m d}$ is the $m d \times m d$ identity matrix, and the infinite-dimensional Hilbert space $\mathcal{H}$ is represented by $\mathbb{R}^{m d}$ with inner product given by the weighted scalar product $\langle \bph{x}, \bph{y} \rangle = \frac{\beta}{m} \bph{x} \bullet \bph{y} $. The functional in \eqref{eq:U_TPS} is discretized as \[
U_{m}( \bph{x}) = \frac{\beta}{m} G_{m}(\bph{x})
\;, \quad \text{where} \quad G_{m}(\bph{x}) =
\sum_{j=1}^{m}  G(\bph{x}_{j+1:j+d}) \;.
\] In summary, the infinite-dimensional path distribution $\mu(dx)$ is approximated by the finite-dimensional distribution $\mu_{m}(d \bph{x}) \propto \exp(-U_{m}(\bph{x}))  \mathcal{N}(0,\frac{m}{\beta} \bph{\mathcal{C}}_{\mathsf{a}})(d \bph{x})$.

In this context, pHMC generates a Markov chain on $\mathbb{R}^{m d}$ with invariant measure $\mu_{m}$ and with transition step given by \begin{equation} \label{eq:transitionstep_PIMD}
\bph{x} \mapsto \bph{X}'(\bph{x}) ~:=~ \bph{q}_T(\bph{x}, \bph{\xi}) \;, \qquad \bph{\xi} \sim \mathcal{N}(0, \frac{m}{\beta} \bph{\mathcal{C}}_{\mathsf{a}})
\end{equation} 
where $\bph{q}_t$ solves \eqref{eq:dynamics_TPS} with $    \bph{b}(\bph{x}) = -  \bph{x} - \bph{\mathcal{C}}_{\mathsf{a}} \nabla  G_{m} (\bph{x})$.  %Let $\pi_{m}$ denote the transition kernel of \eqref{eq:transitionstep_PIMD}. 

%Note that integrability of $e^{-U_{\epsilon}(x)}$ with respect to $\mu_0$ requires that the underlying potential energy $V(q)$ grows faster than $(1/2) \epsilon |q|^2$.
%As in the transition path sampling example, one can choose $\tilde{\mathcal{C}} = \mathcal{C}_{\epsilon}$ and $b(x) = -x - \mathcal{C} DU_{\epsilon}(x)$ and obtain a condition similar to \eqref{A0A_TPS} that involves $\epsilon>0$.

\begin{theorem}[Path Integral Molecular Dynamics]
\label{thm:CONTR_PIMD}
Suppose that Assumption~\ref{G123} holds.
Let $\kappa:=6 \mathsf{a}^{-1} L_G $, 
$m_{\ell}=\lceil  \sqrt{3 L_G/2} (\beta/\pi) \rceil$,
$n=2 m_{\ell} d -d$, and 
 $m^{\star}=\lceil 2 \pi m_{\ell} \rceil$. Let $R$, $c$, $C$ and $\epsilon$ be defined as
\begin{align}
R \ &:= \ 16 \sqrt{20} (1+\kappa)^{3/2} ( (1/2)  \beta \mathsf{a}^{-1} M_G^2 + 2 d (\beta^2 \mathsf{a} + 1) )^{1/2} \;, \label{R_PIMD} \\
 \quad c \ &:= \  \min( (1/32) T^2, (1/128) T^2 \max(R,T) e^{-\max(R,T)/T} ) \;, \label{crate_PIMD} \\
C \ &:= \ \max(2 T a^{1/2}, 23 ((1/2)  \beta \mathsf{a}^{-2} M_G^2 + 2 d (\mathsf{a}^{-1} + \beta^2)  )^{1/2} e^{R/(2 T)})  \;,  \label{C_PIMD} \\
\epsilon \ &:= \  (1/80) \mathsf{a}^2 ( \beta  M_G^2 )^{-1} e^{-R/T} \;. \label{eps_PIMD}  
\end{align}
 Suppose that the duration parameter $T \in (0, R)$ satisfies
\begin{align}
(1+\kappa )^{3/2} T^2 \ &\le\  \min\left( \frac{1}{96 (1+ \kappa)}, \frac{1}{256 (1+\kappa )^{3/2} R^2} \right) \;. \label{A0A_PIMD} 
\end{align}
Then for any $m > m^{\star}$, $k\in \mathbb{N}$, and probability measure $\nu_m$ on  $\mathbb{R}^{m d}$,
\eqref{QBpHMC_TPS} holds for the transition kernel of \eqref{eq:transitionstep_PIMD}.
\end{theorem}

A proof of this result is given in \S\ref{sec:PIMD_proofs}.

\subsection{Numerical Illustration of Couplings}

Before turning to the proofs of our main results, we test the two-scale coupling defined by \eqref{**} numerically on the following distributions.  
\begin{itemize}
\item A TPS distribution with the three-well path potential energy function illustrated in Figure~\ref{fig:coupling_sample_paths} (a).  The initial conditions of the components of the coupling are taken to be paths that pass through the two channels that connect the bottom two wells located at $x^{\pm} \approx (\pm 1.048, -0.042)$. 
\item A PIMD distribution where the underlying potential energy is the negative logarithm of the normal mixture density illustrated in Figure~\ref{fig:coupling_sample_paths} (b). The mixture components are twenty two-dimensional Gaussian distributions with covariance matrix given by the $2 \times 2$ identity matrix and with mean vectors given by $20$ independent samples from the uniform distribution over the rectangle $[0, 10] \times [0,10]$.  The energy barriers are not large.  The potential energy in this example is adapted from \cite{LiWo2001,KoZhWo2006}.   The initial paths are taken to be two unit circles one centered at $(1,1)$ and the other centered at $(9,9)$.  The parameter $\mathsf{a}$ is selected to be $0.1$.
\item A PIMD distribution where the underlying potential energy is the negative logarithm of the Laplace mixture density illustrated in Figure~\ref{fig:coupling_sample_paths} (c).  The mixture components are twenty two-dimensional (regularized) Laplace distributions using the same covariance matrix and
mean vectors as in the preceding example.  However, unlike the preceding example, in this example the underlying potential is only asymptotically convex.  The initial paths are taken to be two unit circles one centered at $(1,1)$ and the other centered at $(9,9)$.
The parameter $\mathsf{a}$ is selected to be $0.1$.
\item A PIMD distribution where the underlying potential energy is the banana-shaped potential energy illustrated in Figure~\ref{fig:coupling_sample_paths} (d). This function is highly non-convex and unimodal with a global minimum at the point $(1,1)$.  This minimum lies in a long, narrow, banana shaped valley.  The initial paths are taken to be small circles centered at $(\pm 4,16)$.  The parameter $\mathsf{a}$ is selected to be $1.0$.
\end{itemize} 
 For the TPS and PIMD distributions we use the finite-dimensional approximations described in \S\ref{sec:TPS} and~\S\ref{sec:PIMD}, respectively.   The resulting semi-discrete evolution equations are discretized in time using a strongly stable symmetric splitting integrator \cite{BePiSaSt2011,KoBoMi2019}.  We describe this integrator in the specific context of TPS, since a very similar method is used for PIMD with
 $\bph{\mathcal{C}}_{\mathsf{a}}$ replacing $\bph{\mathcal{C}}$ in the dynamics, and 
 the covariance matrix $(m/\beta) \bph{\mathcal{C}}_{\mathsf{a}}$ replacing $((m+1)/\tau) \bph{\mathcal{C}}$ in the velocity randomization step.  First, split \eqref{eq:dynamics_TPS} with $\bph{b}(\bph{x}) = \bph{x}-\bph{\mathcal{C}} \nabla G_{m}(\bph{x})$ into \begin{equation} \label{eq:AA}
 (\mathrm{A}):\qquad    \dot{\bph{q}}_t = \bph{v}_t \;, \quad \dot{\bph{v}}_t = -\bph{q}_t
\end{equation}
\begin{equation} \label{eq:BB}
(\mathrm{B}):\qquad    \dot{\bph{q}}_t = 0 \;, \quad \dot{\bph{v}}_t = -\bph{\mathcal{C}} \nabla G_{m}(\bph{q}_t)
\end{equation}
with corresponding flows explicitly given by
\begin{center}
\begin{tabular}{c}
$\varphi_t^{(\mathrm{A})}(\bph{q}_0,\bph{v}_0) = 
(\cos(t) \bph{q}_0 + \sin(t) \bph{v}_0 \;,  -\sin(t) \bph{q}_0 + \cos(t) \bph{v}_0)$ \\
$\varphi_t^{(\mathrm{B})}(\bph{q}_0,\bph{v}_0) = (\bph{q}_0, \bph{v}_0 - t \bph{\mathcal{C}} \nabla G_{m}(\bph{q}_0))$.
\end{tabular}
\end{center}
Given a time step size $\Delta t>0$, and using these exact solutions, a $\Delta t$ step of the symmetric splitting integrator we use is given by \begin{equation} \label{psi_dt}
\psi_{\Delta t} = \varphi_{\Delta t/2}^{(\mathrm{B})} \circ \varphi_{\Delta t}^{(\mathrm{A})} \circ \varphi_{\Delta t/2}^{(\mathrm{B})} \;.
\end{equation}  In order to mitigate the effect of periodicities or near periodicities in the underlying dynamics, we choose the number of integration steps to be geometrically distributed with mean $T/\Delta t$.  The idea of duration randomization has a long history \cite{duane1985stochastic,duane1986theory, Ma1989, CaLeSt2007, BoSa2016, BoSaActaN2018}. The initial velocity is taken to be an $m d$-dimensional standard normal vector with covariance matrix $((m+1)/\tau) \bph{\mathcal{C}}$ and a Metropolis accept/reject step is added to ensure the algorithm leaves invariant $\mu_m$ \cite{tierney1998note, BoSaActaN2018}.  In summary, we use the following transition step in the simulations.

%\psi_h = \varphi_{(1/2)h}^{(B)}\circ \varphi_{h}^{(A)}\circ\varphi_{(1/2)h}^{(B)}

\begin{algorithm}[Numerical Randomized pHMC] \label{algo:numerical_rhmc}
Denote by $T>0$ the duration parameter and let $\psi_{\Delta t}$ be the time integrator described in \eqref{psi_dt}.  Given the current state of the chain $\bph{x} \in \mathbb{R}^{m d}$, the algorithm outputs the next state of the chain $\bph{X} \in \mathbb{R}^{m d}$ as follows.
\begin{description}
\item[Step 1] Generate a $d$-dimensional random vector $\bph{\xi} \sim \mathcal{N}(\bph{0},((m+1)/\tau) \bph{\mathcal{C}})$.
\item[Step 2] Generate a geometric random variable $k$ supported on the set $\{1, 2, 3, ... \}$ with mean $T/ \Delta t$.
\item[Step 3] Output $\bph{X} = \gamma \tilde{\bph{q}}_k + (1-\gamma) \bph{x}$ where $(\tilde{\bph{q}}_k, \tilde{\bph{v}}_k) = \psi_{\Delta t}^k(\bph{x}, \bph{\xi})$, and given $\bph{\xi}$ and $k$, $\gamma$ is
    a Bernoulli random variable with parameter $\alpha$ defined as \[
    \alpha = \min\{ 1, \exp\left(- [ \mathcal{E}(\tilde{\bph{q}}_k, \tilde{\bph{v}}_k) - \mathcal{E}(\bph{x},\bph{\xi}) ]\right) \}
    \] where 
    $\mathcal{E}(\bph{x},\bph{v}) = (1/2) \langle \bph{v}, \bph{\mathcal{C}}^{-1} \bph{v} \rangle    + U_{m}(\bph{x})+  (1/2) \langle \bph{x} , \bph{\mathcal{C}}^{-1} \bph{x} \rangle$.
\end{description}
\end{algorithm}

We stress that $\bph{\xi}$ and $k$ from (Step 1) and (Step 2) are mutually independent and independent of the state of the Markov chain associated to pHMC.  We pick the time step size $\Delta t$ of the integrator sufficiently small to ensure that $99\%$ of proposal moves are accepted on average in (Step 3).

%For simplicity, we apply the coupling in \eqref{**} globally, we approximate the continuous-space Dirichlet and periodic Laplacians by finite differences, and we apply the special splitting integrator introduced in \cite{BePiSaSt2011} operated at a step size $\Delta t$ sufficiently small to ensure that $99\%$ of proposed moves are accepted.  

Realizations of the coupling process are shown in Figure~\ref{fig:coupling_sample_paths}.   We chose parameters only for visualization purposes.  The different components of the coupling are shown as different color dots.  The insets of the figures show the distance between the components of the coupling as a function of the number of steps.

Figure~\ref{fig:coupling_times} shows the average time after which
 the distance between the components of the coupling is for the first time within $10^{-12}$.  To produce this figure, we generated one hundred samples of the coupled process for one hundred different values of the duration parameter $T$.  As indicated in the figure legends, the coupling parameter $\gamma$ is set equal to either:  
 \begin{itemize}
\item $\gamma=0$ which corresponds to a synchronous coupling of the initial velocities;
\item $\gamma=T^{-1}$ which corresponds to the optimal coupling of the initial velocities when $\bph{b}(\bph{x})=0$; and,
\item $\gamma=\cot(T)$ which corresponds to the optimal coupling of the initial velocities when $\bph{b}(\bph{x})=-\bph{x}$.
\end{itemize}

%%%%%%%%%%%%%%%%%%%%%%%%%%%%%%%%%%%%%%%%%%%%%%%%%%%%%%%%%%%%%%%%%%%%%%%%%

\begin{comment}
\begin{figure}[t]
\begin{center}
\includegraphics[width=0.3\textwidth]{phmc_coupling_snapshot_1_nmm.pdf} 
\hspace{0.125in}
\includegraphics[width=0.3\textwidth]{phmc_coupling_snapshot_5_nmm.pdf} 
\hspace{0.125in}
\includegraphics[width=0.3\textwidth]{phmc_coupling_snapshot_10_nmm.pdf}
\hbox{
            (a) 
        \hspace{1.425in} (b) 
        \hspace{1.425in} (c) 
        } 
\end{center}
\caption{ \small
 }
 \label{fig:coupling_snapshots}
\end{figure}
\end{comment}

\begin{figure}[t]
\begin{center}
\includegraphics[width=0.2\textwidth]{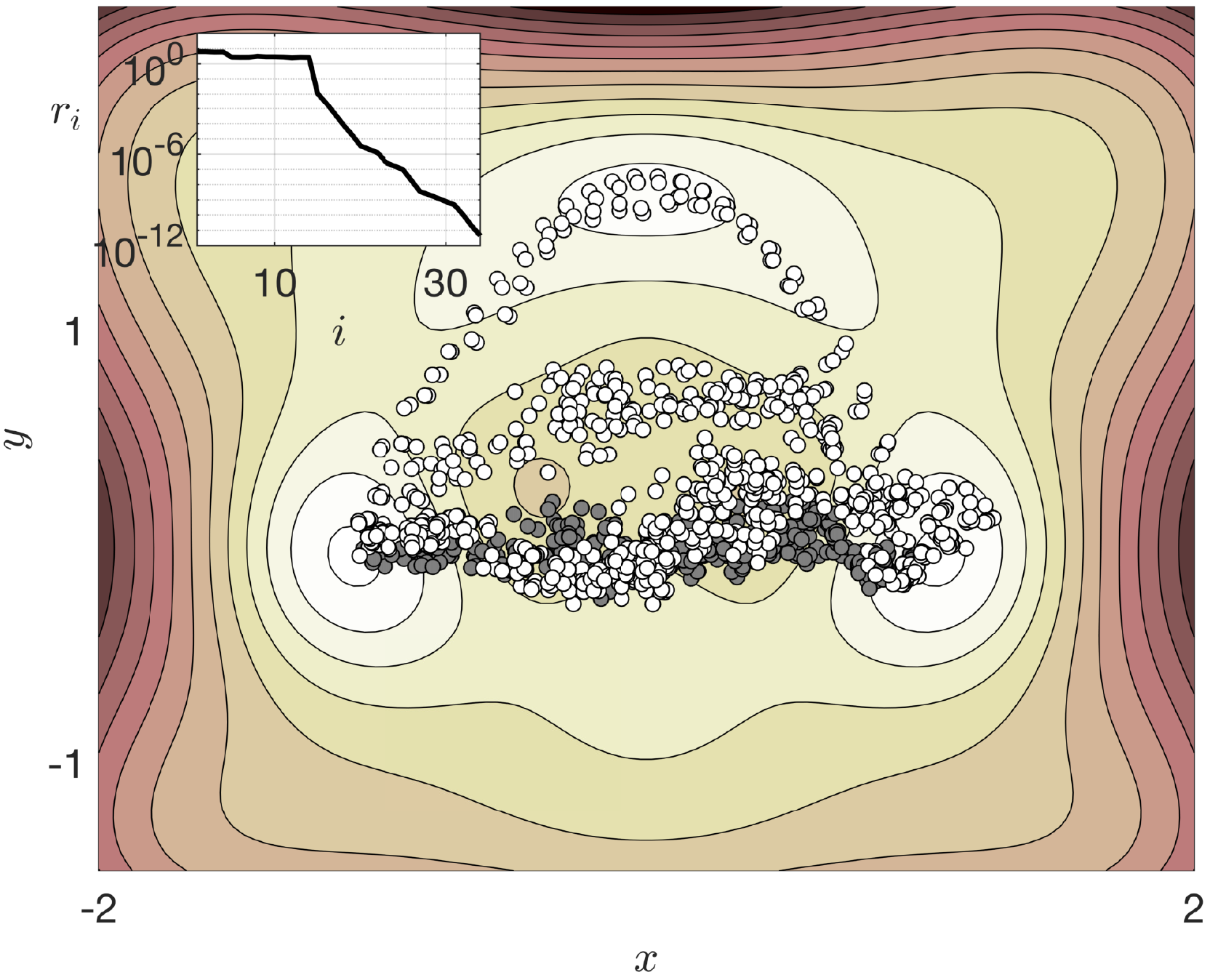} 
\hspace{0.125in}
\includegraphics[width=0.2\textwidth]{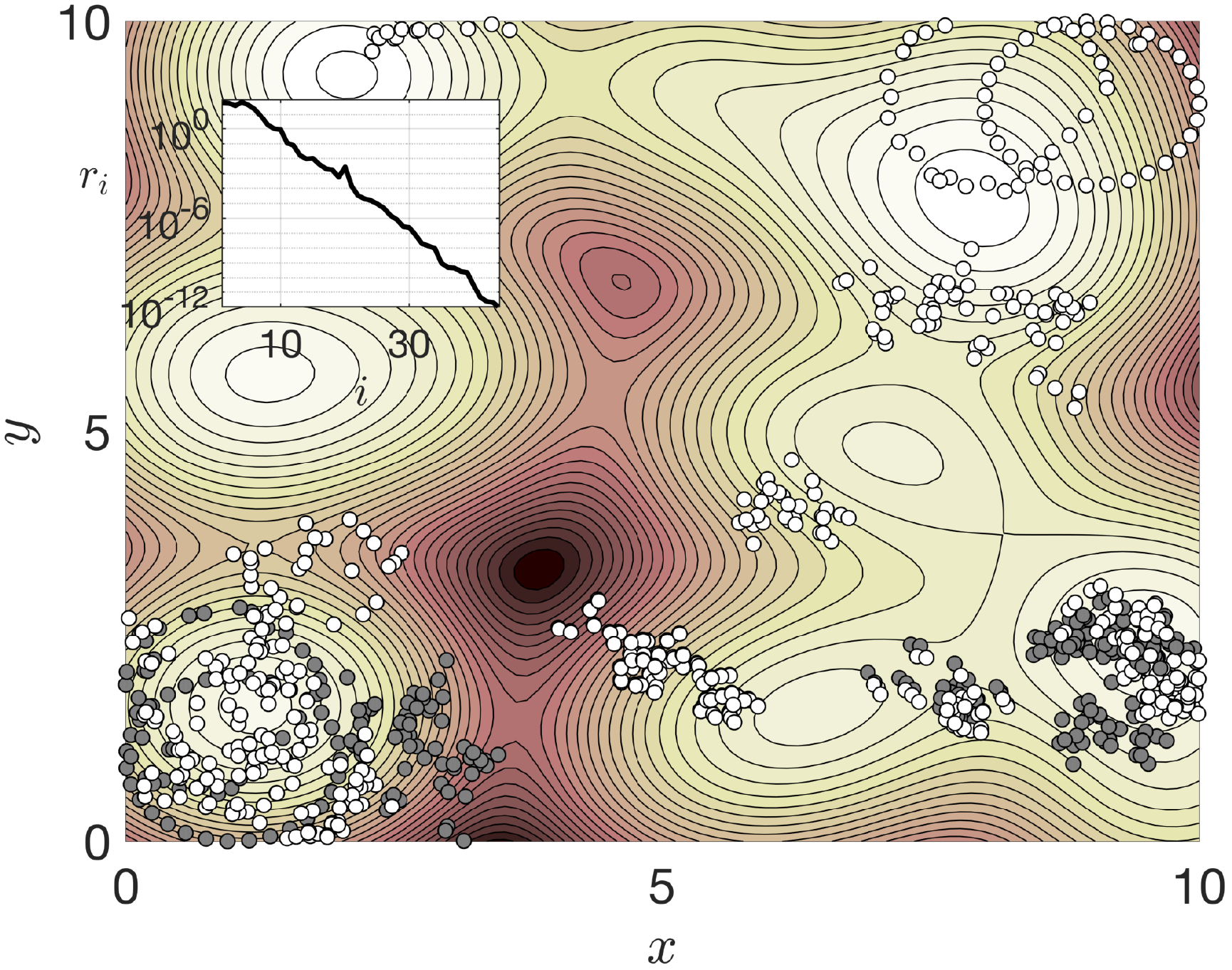} 
\hspace{0.125in}
\includegraphics[width=0.2\textwidth]{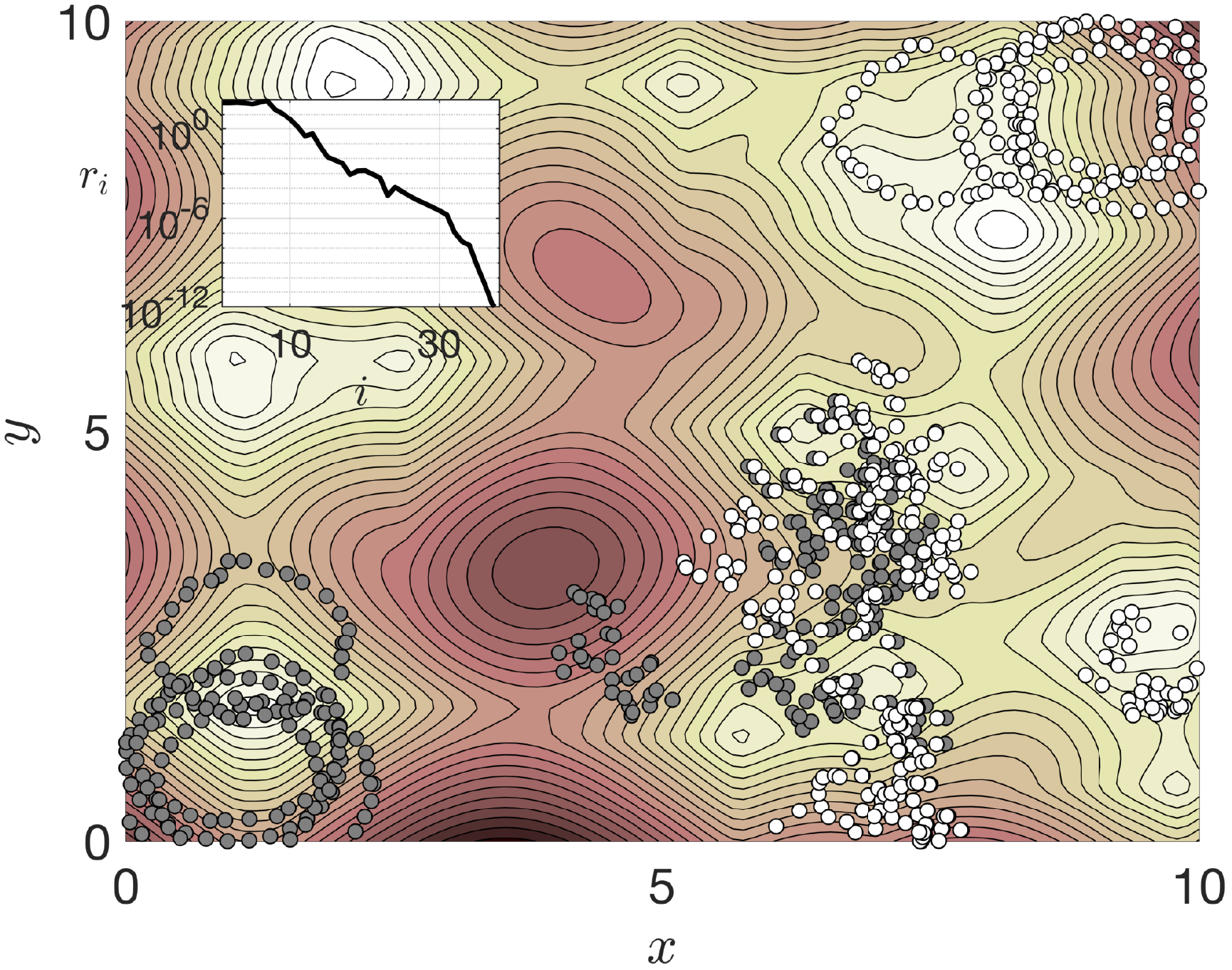}
\hspace{0.125in}
\includegraphics[width=0.2\textwidth]{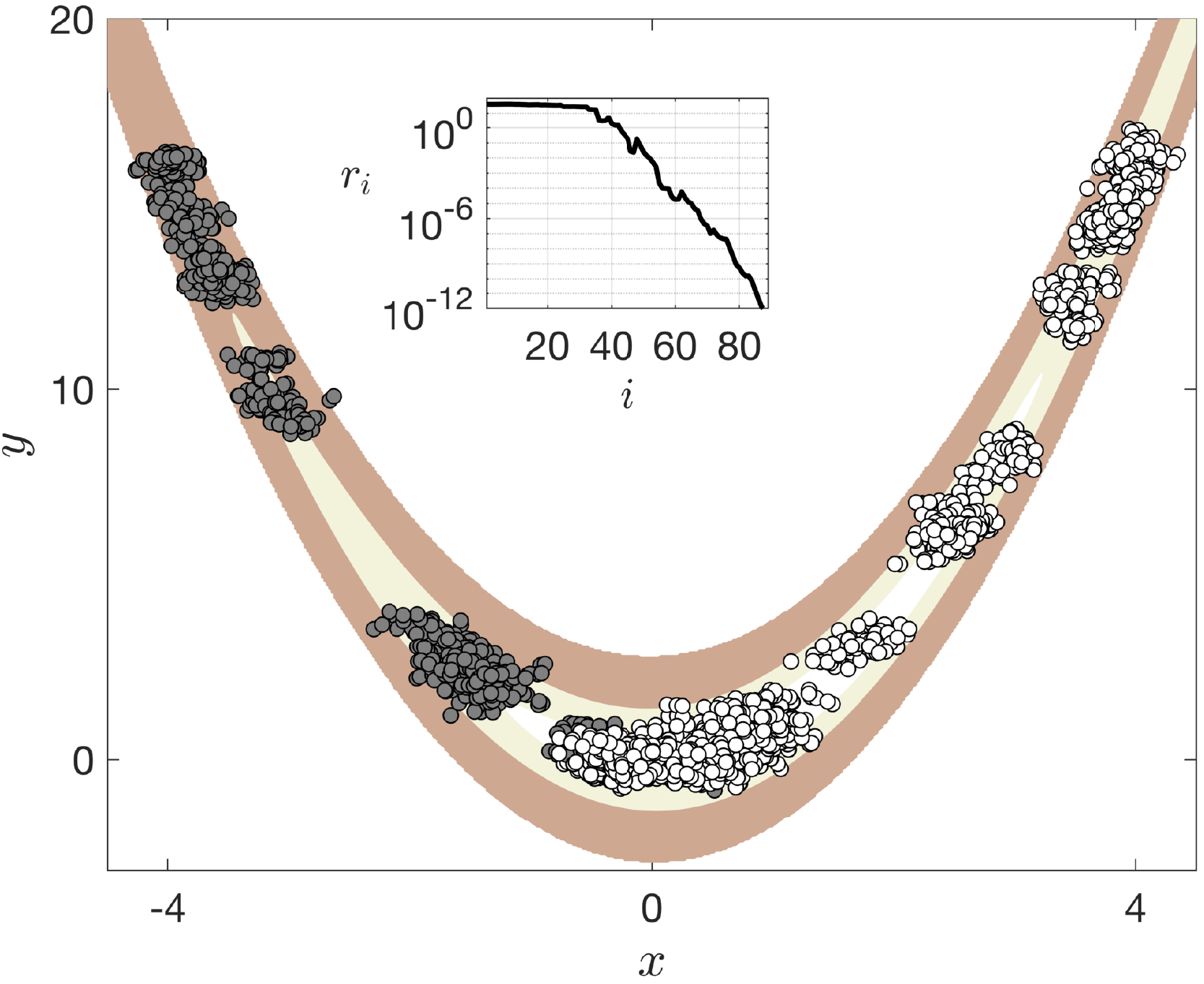}
\hbox{
            (a) 
        \hspace{1.in} (b) 
        \hspace{1.in} (c) 
         \hspace{1.in} (d) 
        } 
\end{center}
\caption{ \small {\bf Realizations of Two-Scale Coupling.}
This figure illustrates realizations of the coupling with $T=\gamma^{-1}$.   The different components of the coupling are shown as different color dots.  
%The size of the dots is related to the number of steps: points along the trajectory corresponding to a larger number of steps have smaller markers.  
A contour plot of the underlying potential energy function is shown in the background.
The inset plots the distance $r_i$ between the components of the coupling as a function of the step index $i$. 
The simulation is terminated when this distance first reaches $10^{-12}$.  
 In (a), (b), (c), and (d), this occurs in 34, 44, 38, and 88 steps, respectively.
 }
 \label{fig:coupling_sample_paths}
\end{figure}

\begin{figure}[t]
\begin{center}
\includegraphics[width=0.2\textwidth]{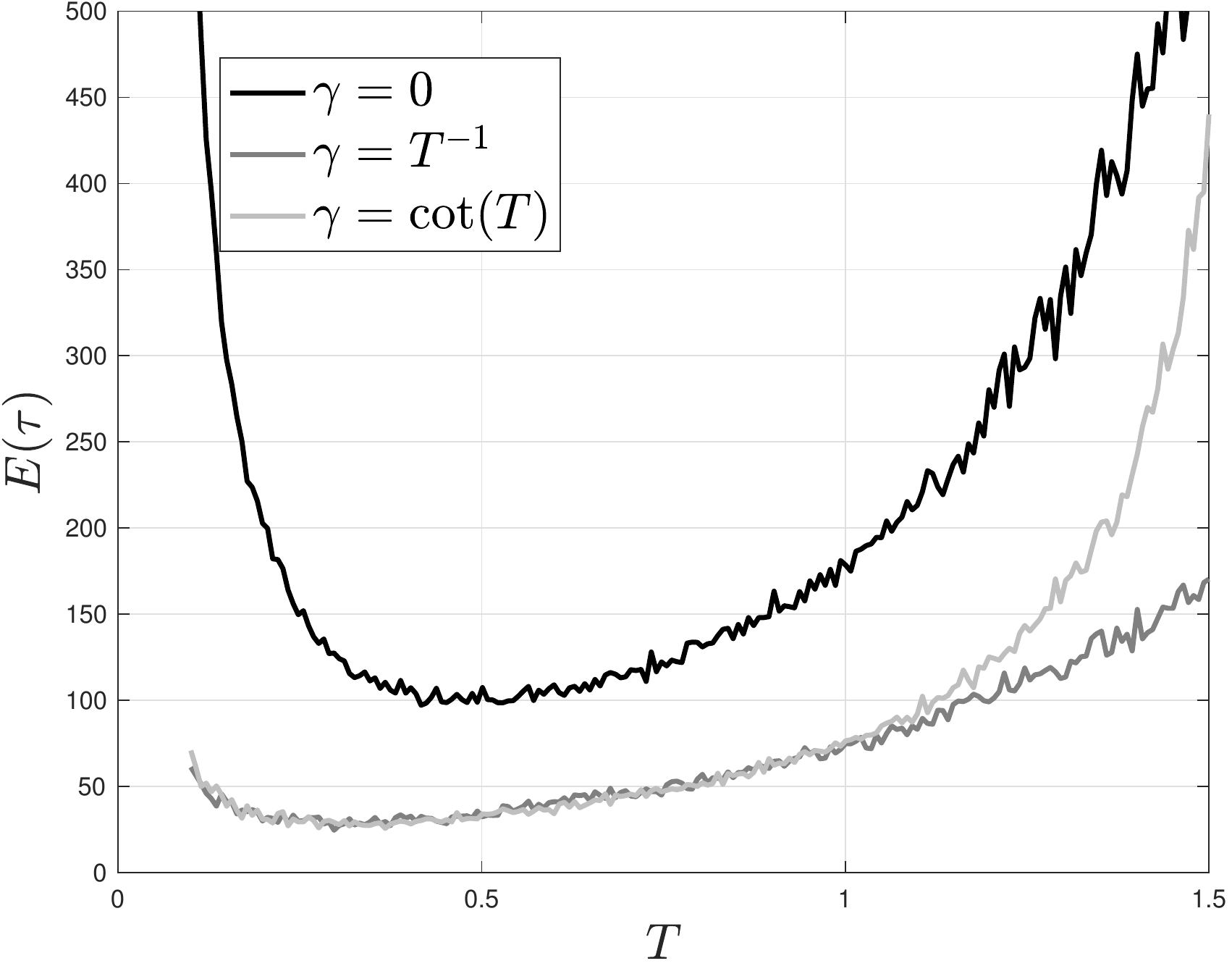} 
\hspace{0.125in}
\includegraphics[width=0.2\textwidth]{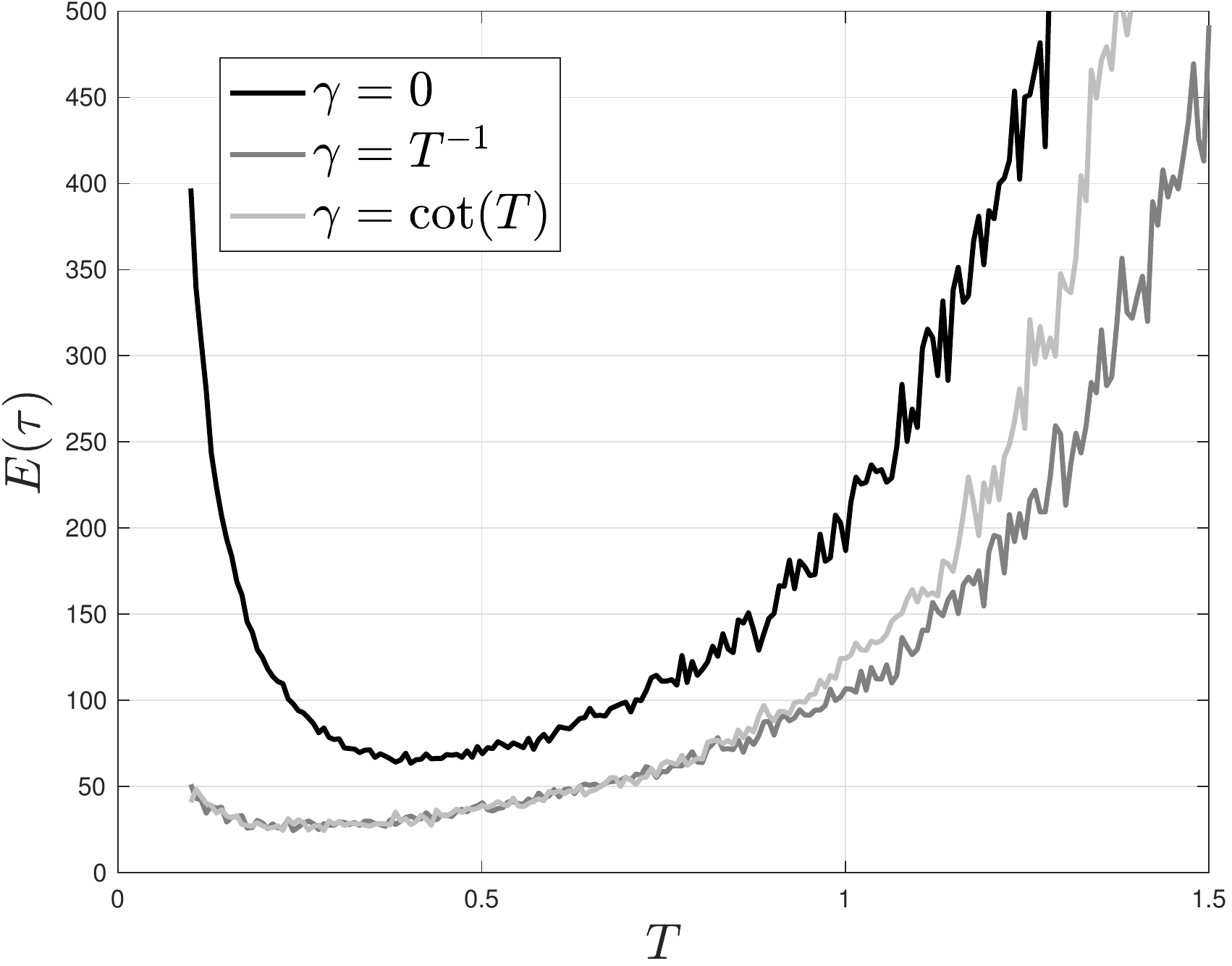} 
\hspace{0.125in}
\includegraphics[width=0.2\textwidth]{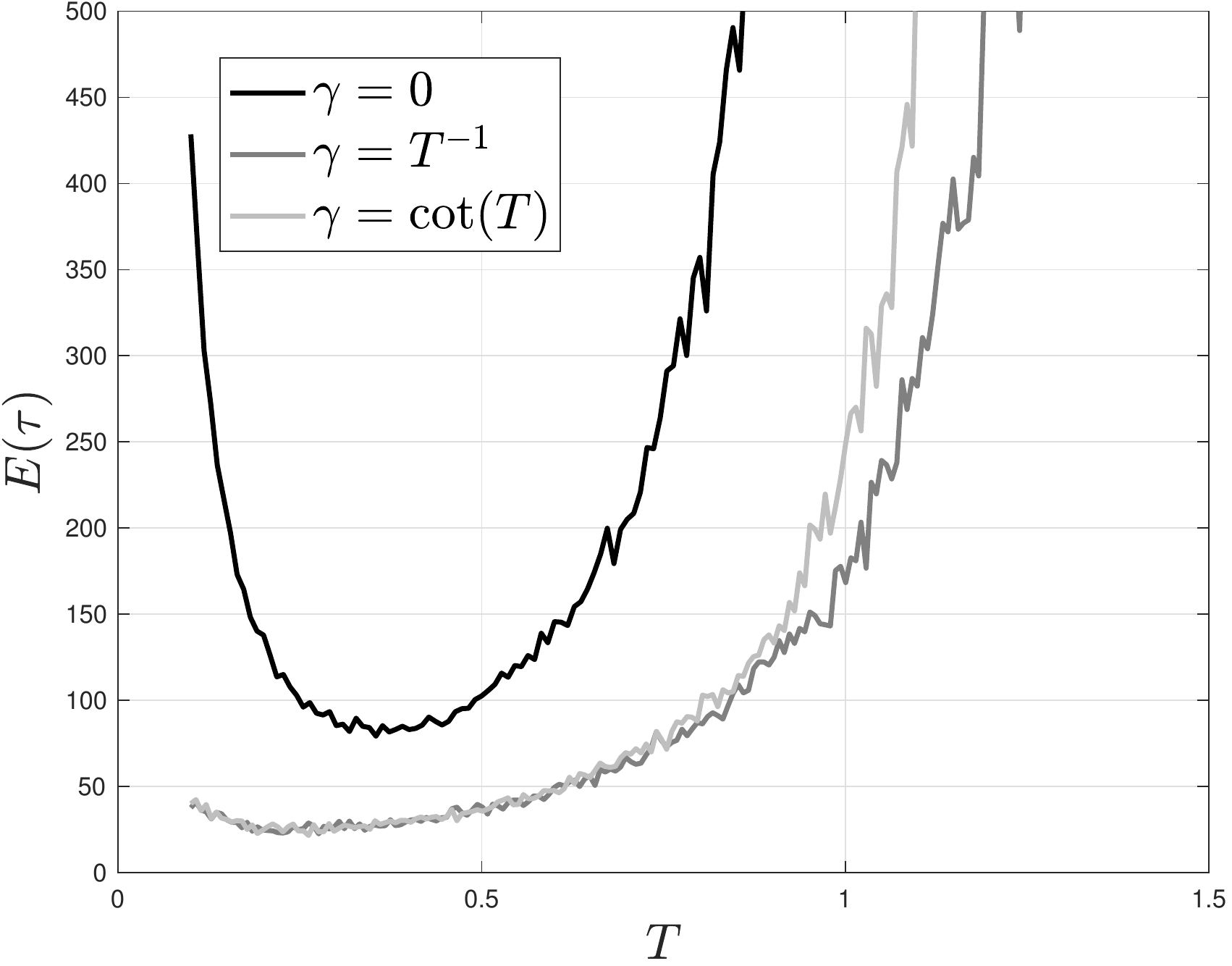}
\hspace{0.125in}
\includegraphics[width=0.2\textwidth]{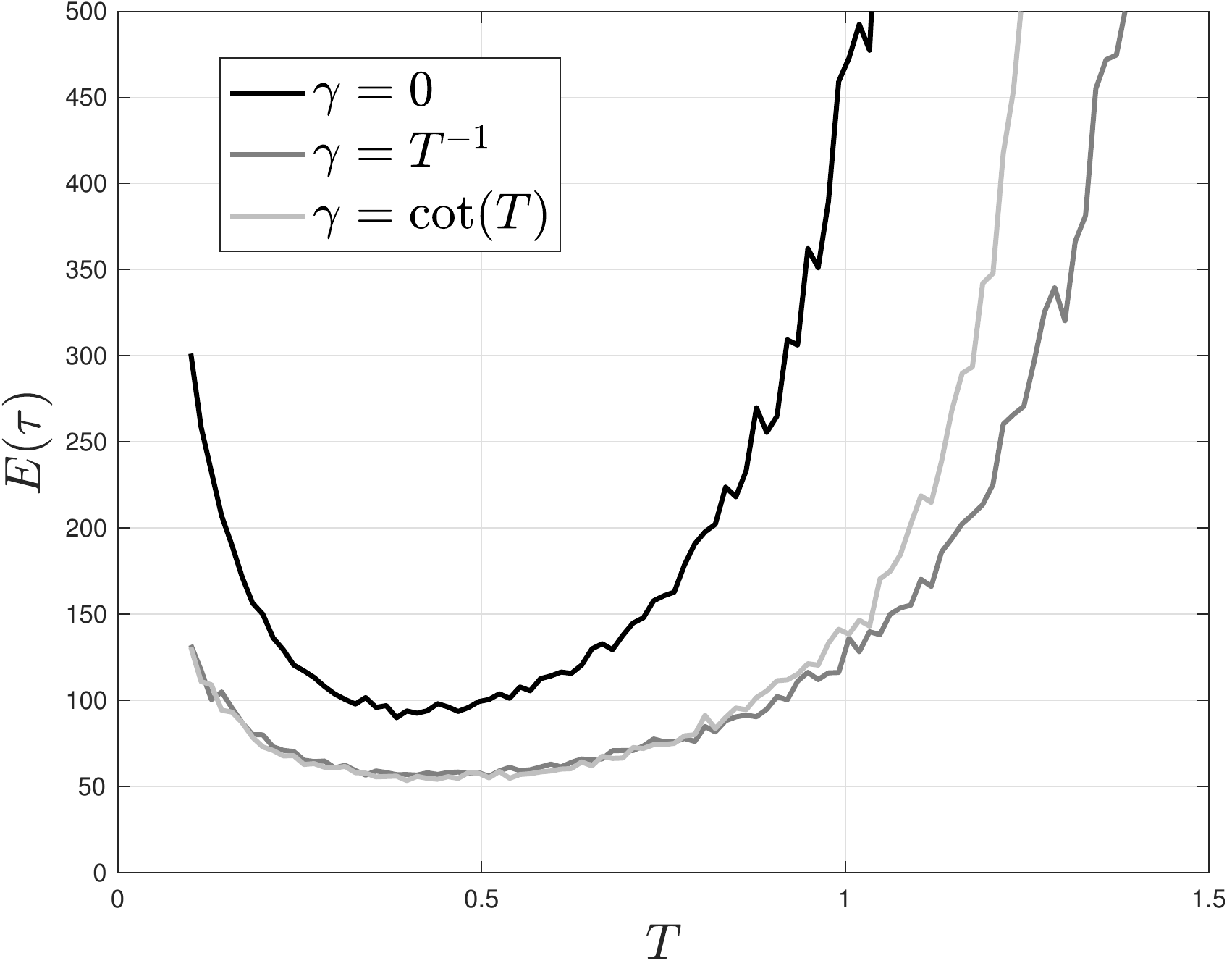}
\hbox{
            (a) 
        \hspace{1.in} (b) 
        \hspace{1.in} (c) 
         \hspace{1.in} (d) 
        } 
\end{center}
\caption{ \small {\bf Mean Coupling Times.}
This figure illustrates the average of the random time $\tau$ after which 
 the distance between the components of the coupling is for the first time within $10^{-8}$. The estimated average is plotted
as a function of the duration $T$ of the Hamiltonian dynamics
for $\gamma=0$ (black), $\gamma=T^{-1}$ (gray), and $\gamma=\cot(T)$ (light gray).    In all cases, note that the minimum of the function
is smaller and occurs at a smaller value of $T$ when $\gamma=T^{-1}$ or $\gamma=\cot(T)$, and the difference between the minima for $\gamma=T^{-1}$ and $\gamma=\cot(T)$ is slight, because these minima occur at $T\le1$ where $\cot(T) \approx T^{-1}$.
 }
 \label{fig:coupling_times}
\end{figure}

%%%%%%%%%%%%%%%%%%%%%%%%%%%%%%%%%%%%%%%%%%%%%%%%%%%%%%%%%%%%%%%%%%%%%%%%

\section{A priori bounds}\label{sec:apriori}

In this section we gather several bounds for the dynamics and for the coupling that will be crucial in the proof of our main result. 

\subsection{Bounds for the dynamics}\label{sec:main_results:boundsfordynamics} 

In the following, we assume throughout that Assumption \ref{B123} is satisfied, and
\begin{eqnarray}
	\L t^2 &\leq& 1. \label{0}
\end{eqnarray}
Recall that $\phi_t=(q_t,v_t)$ denotes the flow of \eqref{eq:dynamicsgeneral}.  With the exception of using a different norm, the proofs of Lemmas~\ref{0A} and~\ref{1A} below are identical to the proofs of Lemmas~3.1 and 3.2 in Ref.~\cite{BoEbZi2018} and therefore not repeated.

\begin{lemma}\label{0A}
For any $x,v\in\H^s$,
\begin{eqnarray}
\label{6}
\ \ \sup_{r\le t} \snorm{q_r(x,v)-(x+rv)}&\leq& \L t^2
\max(\snorm{x},\snorm{x+tv}),\quad\mbox{and}
\\ \label{7}
\sup_{r\leq t} \snorm{v_r(x,v)-v} &\leq &
\L t\, \sup_{r\le t}\snorm{q_r(x,v)}\\
\nonumber &\le &
\L t (1+\L t^2) \max(\snorm{x},\snorm{x+tv}).
\end{eqnarray}
In particular,
\begin{eqnarray}\label{8}
\sup_{r\leq t}\snorm{q_r(x,v)} &\leq& 2 \max(\snorm{x},\snorm{x+tv}),\qquad\mbox{and}\\
\label{9}
\sup_{r\leq t}\snorm{v_r(x,v)} &\leq&
\snorm{v} + 2 \L t  \max(\snorm{x},\snorm{x+tv}).
\end{eqnarray}
\end{lemma}

\begin{lemma}\label{1A}
	For any $x,y,u,v\in\H^s$,
	\begin{eqnarray}		
\lefteqn{\sup_{r\leq t} \snorm{q_r(x,u)-q_r(y,v)-(x-y)-r(u-v)}}\nonumber\\
\label{10} 		& \leq &   \L t^2 \max\left(\snorm{x-y},\snorm{(x-y)+t(u-v)}\right),\qquad\mbox{and}
	\\
\nonumber\lefteqn{	\sup_{r\leq t} \snorm{v_r(x,u)-v_r(y,v)-(u-v)} 
\ \le\ \L t\, \sup_{r\le t} \snorm{q_r(x,u)-q_r(y,v)}}
\\ &\leq  &  
		\L t (1+\L t^2)\max\left(\snorm{x-y},\snorm{(x-y)+t(u-v)}\right) .\label{11}
	\end{eqnarray}
	In particular,
	\begin{eqnarray}
		\label{12}
\sup_{r\leq t} \snorm{q_r(x,u)-q_r(y,v)} \ \leq \  (1 + \L t^2) \max(\snorm{x-y},\snorm{(x-y)+t(u-v)}).
	\end{eqnarray}
\end{lemma}

Lemma~\ref{0A} is used in the proof of the Foster-Lyapunov drift condition in Lemma \ref{lem:LYAP}.  Lemma~\ref{1A} is used in
the proof of Lemma~\ref{lem:contr} below.

\subsection{Bounds related to two-scale coupling}

The following lemma is used in the proof of Theorem \ref{thm:CONTR} to obtain a contraction for the two-scale coupling when the distance between the components of the coupling is sufficiently small, i.e.,  $\anorm{x-y}<R$.

\begin{lemma} \label{lem:contr}
Suppose that $\gamma >0$ and $t>0$ satisfy $\gamma t \le 1$ and $\L t^2 \le 1/4$.  Then for any $x, y, u, v \in \H^s$ such that $v^h = u^h$ and $v^{\ell} = u^{\ell} + \gamma (x^{\ell} - y^{\ell})$, we have
\begin{description}
\item[(i)] $\snorm{\tilde{\mathcal{C}}^{-\frac{1}{2}} \mathcal{C}^{\frac{s}{2}}  (q_t^{\ell}(x,u) - q_t^{\ell}(y,v) )} \le \\ \left(1-\gamma t + \dfrac{5}{8} \dfrac{\sigma_{max}}{\sigma_{min}}  L t^2 \right) \snorm{ \tilde{\mathcal{C}}^{-\frac{1}{2}} \mathcal{C}^{\frac{s}{2}}  (x^{\ell}-y^{\ell}) } + \dfrac{5}{8} \sigma_{max} L t^2 \snorm{x^h-y^h}$
\smallskip
\item[(ii)] $\snorm{ q_t^h(x,u) - q_t^h(y,v)} \le \left(1 - \dfrac{1}{4} t^2  \right) \snorm{x^h - y^h} + \dfrac{1}{4} \sigma_{min}^{-1} t^2   \snorm{ \tilde{\mathcal{C}}^{-\frac{1}{2}} \mathcal{C}^{\frac{s}{2}} (x^{\ell}-y^{\ell})}  $
\end{description}
\end{lemma}

\begin{proof}
Let $\mathcal{G}(x) = b(x) + x$, $z_t = q_t(x,u) - q_t(y,v)$, and $w_t = \frac{dz_t}{dt}$. By \eqref{eq:dynamicsgeneral}, \begin{align*}
	&\frac{d}{dt} z^{\ell}_t \ = \ w^{\ell}_t,\quad \frac{d}{dt} w^{\ell}_t \ = \  b^{\ell}(q_t(x,u)) - b^{\ell}(q_t(y,v)), \\
	&\frac{d}{dt} z^h_t \ = \ w^h_t,\quad \frac{d}{dt} w^h_t \ = \ -z^h_t + \mathcal{G}^h(q_t(x,u)) - \mathcal{G}^h(q_t(y,v)),
\end{align*}
with $w_0 = - \gamma z^{\ell}_0$. These 
are second order linear ordinary differential equations, perturbed
by a nonlinearity. A variation of constants ansatz shows that they are
equivalent to the equations
  \begin{align*}
	& z^{\ell}_t \ = \ (1-\gamma t)  z^{\ell}_0+ \int_0^t (t-r) \left( b^{\ell}(q_r(x,u)) - b^{\ell}(q_r(y,v)) \right) dr, \\
	& z^h_t \ = \ \cos(t) z^h_0 + \int_0^t \sin(t-r) \left( \mathcal{G}^h(q_r(x,u)) - \mathcal{G}^h(q_r(y,v)) \right) dr.
\end{align*}
Since $t^2\le \L t^2 \le 1/4$, $\gamma t \le 1$,
and by Assumption~\ref{B123} (B1) and (B2) and \eqref{compnorm}, 
\begin{align}
	 \snorm{\tilde{\mathcal{C}}^{-\frac{1}{2}} \mathcal{C}^{\frac{s}{2}}  z^{\ell}_t} \ &\le \ (1-\gamma t) \snorm{\tilde{\mathcal{C}}^{-\frac{1}{2}} \mathcal{C}^{\frac{s}{2}}  z^{\ell}_0}+ \sigma_{max} \L \int_0^t (t-r) \snorm{z_r} dr, \nonumber \\
	\ &\le  \ (1-\gamma t) \snorm{\tilde{\mathcal{C}}^{-\frac{1}{2}} \mathcal{C}^{\frac{s}{2}} z^{\ell}_0}+ \sigma_{max} \frac{ \L t^2}{2} \sup_{r \le t} \snorm{z_r}, \label{ieq:zl} \\
	 \snorm{z^h_t} \ &\le \ |\cos(t)| \snorm{z^h_0} + \frac{1}{3} \int_0^t |\sin(t-r)| \snorm{z_r} dr, \nonumber \\
	\ &\le  \left(1-\frac{t^2}{2} +\frac{t^4}{24}  \right) \snorm{z^h_0} + \frac{t^2}{6}  \sup_{r \le t} \snorm{z_r}, \label{ieq:zh} 
\end{align}
Here in \eqref{ieq:zh} we used that $t\le 1/2$, and hence $|\cos(t)| \le 1- (1/2) t^2 + (1/24) t^4$, and $|\sin (t-r)|\le t$. Since $w_0=-\gamma z_0^\ell$ and $\gamma t\le 1$,  Lemma~\ref{1A} and \eqref{compnorm} imply 
\begin{align*}
\sup_{r \le t} \snorm{z_r} \le \frac{5}{4} (  \snorm{z_0^{\ell}} + \snorm{z_0^h} ) \le \frac{5}{4} \left( \sigma_{min}^{-1} \snorm{\tilde{\mathcal{C}}^{-\frac{1}{2}} \mathcal{C}^{\frac{s}{2}} z_0^{\ell}} + \snorm{z_0^h} \right)  \;.
\end{align*}
Inserting this estimate into \eqref{ieq:zl} and \eqref{ieq:zh}, and again using $t^2 \le 1/4$ yields, \begin{align*}
	& \snorm{\tilde{\mathcal{C}}^{-\frac{1}{2}} \mathcal{C}^{\frac{s}{2}} z^{\ell}_t} \ \le \ \left(1-\gamma t + \frac{5}{8} \frac{\sigma_{max}}{\sigma_{min}}  L t^2 \right) \snorm{\tilde{\mathcal{C}}^{-\frac{1}{2}} \mathcal{C}^{\frac{s}{2}} z^{\ell}_0}+ \frac{5}{8} \sigma_{max} L t^2 \snorm{z_0^h}, \\
	& \snorm{z^h_t} \ \le \ \left(1- \frac{1}{4} t^2 \right) \snorm{z^h_0} + \frac{1}{4} \sigma_{min}^{-1}  t^2    \snorm{\tilde{\mathcal{C}}^{-\frac{1}{2}} \mathcal{C}^{\frac{s}{2}} z_0^{\ell}}  \qquad \text{as required.}
\end{align*}
\end{proof}

Recall that the two-scale coupling
that we consider ensures that $\xi^{\ell} -\eta^{\ell} =-\gamma z^{\ell}$ 
with the maximal possible probability, where $z=x-y$. The following lemma enables us to control the probability that $\xi^{\ell} -\eta^{\ell} \neq -\gamma z^{\ell}$ for small distances $\anorm{z}<R$.

\begin{lemma}\label{lem:X}
For any choice of $\gamma$,
$ P[\xi^{\ell}-\eta^{\ell}\neq -\gamma z^{\ell}] \le   \snorm{ \gamma \tilde{\mathcal{C}}^{-\frac{1}{2}} \mathcal{C}^{\frac{s}{2}} z^{\ell} }/\sqrt{2\pi} \;.$
\end{lemma}

Since $\tilde{\mathcal{C}} \mathcal{C}^{-s}$ is a trace class operator on $\H^s$,  note that $\tilde{\mathcal{C}}^{-\frac{1}{2}} \mathcal{C}^{\frac{s}{2}}$ is not a bounded operator.   Nonetheless, the bound appearing in Lemma~\ref{lem:X} is finite because the operator $\tilde{\mathcal{C}}^{-\frac{1}{2}} \mathcal{C}^{\frac{s}{2}}$ appearing in that bound only acts on $z^{\ell}$, i.e., the 
$n$-dimensional projection of $z$ onto the finite dimensional space $\H^{s,\ell}$.

\begin{proof}
Recall from \eqref{eq:TVnormal} that 
$P[\xi^{\ell}-\eta^{\ell}\neq -\gamma z^{\ell}] =d_{\TV}(\mathcal{N}(0,  \tilde{\mathcal{C}}) , \mathcal{N}(\gamma z^{\ell}, \tilde{\mathcal{C}}) )$. Let $\tilde{\mathcal{C}}^{\ell}$ denote the restriction of 
$\tilde{\mathcal{C}}$ to $\H^\ell$. Then since $z^\ell\in\H^\ell$, and
by scale invariance of the total variation distance,
\begin{eqnarray*}
\lefteqn{P[\xi^{\ell}-\eta^{\ell}\neq -\gamma z^{\ell}]\ =\ d_{\TV}(\mathcal{N}(0,  \tilde{\mathcal{C}}^\ell ) , \mathcal{N}(\gamma z^{\ell}, \tilde{\mathcal{C}}^\ell ) )}\\ 
&=& d_{\TV}(\mathcal{N}(0,  I_\ell) , \mathcal{N}(\gamma \tilde{\mathcal{C}}^{-1/2}z^{\ell}, I_\ell) )
\ =\ d_{\TV}(\mathcal{N}(0,  1) , \mathcal{N}(|\gamma \tilde{\mathcal{C}}^{-1/2}z^{\ell}|, 1) )\\
&= &2\, \mathcal{N}(0,1)  \left[ (0 ,\,  |\gamma \tilde{\mathcal{C}}^{-1/2} z^{\ell}/2 |)  \right]
\ \le\  |\gamma \tilde{\mathcal{C}}^{-1/2} z^{\ell} |/\sqrt{2\pi},
\end{eqnarray*}
see Figure \ref{fig:tv} for the last equation.
 \end{proof}
\begin{figure}[t]
\begin{center}
\includegraphics[width=0.35\textwidth]{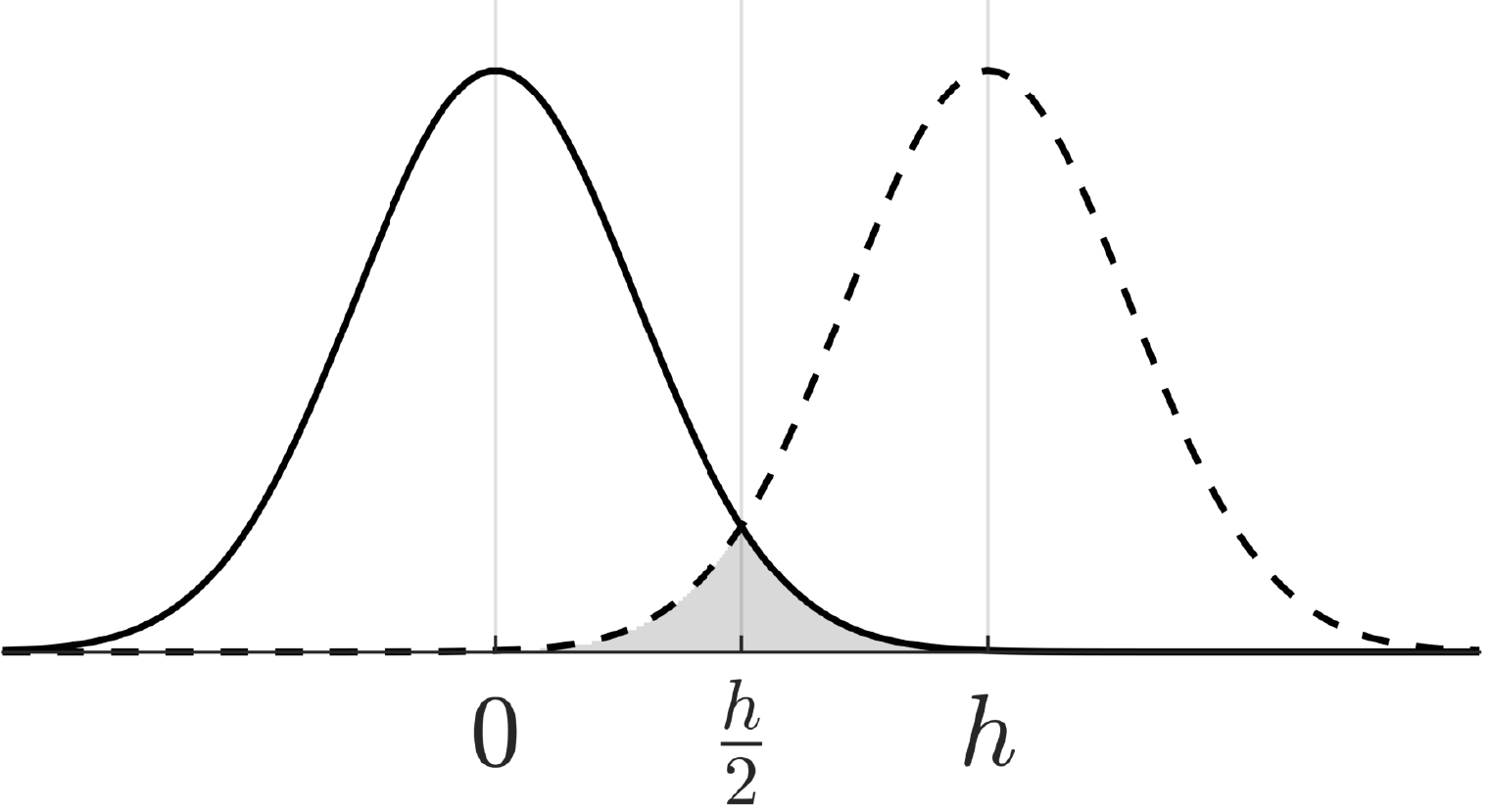} 
\end{center} \vspace{-0.2in}
\caption{ \small The total variation distance between the
one-dimensional normal distributions $\mathcal{N}(0,1)$ and $\mathcal{N}(h,1)$ equals one minus the area of the
shaded region. Therefore, $d_{\TV}(\mathcal{N}(0,  1) , \mathcal{N}(h,1))
=2\mathcal{N}(0,1)[(0,h/2)]$.}
 \label{fig:tv}
\end{figure}

Next we gather some elementary inequalities on the function $f$ in \eqref{eq:f} needed in the proof of Theorem~\ref{thm:CONTR}.  To state these results, let $f_{-}^{\prime}$ denote the left derivative of $f$ which satisfies \[
f_{-}^{\prime}(r) = \begin{cases} e^{-a r} & \text{for $r \le R$} \;, \\ 
0 & \text{for $r > R$}  \;.
\end{cases}
\] 

\begin{lemma} \label{lem:f}
For all $r,\tilde r>0$, the function $f$ in \eqref{eq:f} satisfies:
\begin{description}
\item[(i)] 
$f(\tilde r) - f(r) \le f^{\prime}_{-}(r) (\tilde r-r) \;.$ \smallskip
\item[(ii)] $f(\tilde r) - f(r)  \le  a^{-1} f'_{-}(r) \;.$ \smallskip
\item[(iii)] If $r \le  R$ then \ $
 \max(1,a R) e^{-\max(1,a R)} \le  {r  f^{\prime}_{-}(r)}/{f(r)} \le 1 \;.$
\end{description}
\end{lemma}

\begin{proof}
Property (i) follows from the fact that $f$ is concave.  
Since $f$ is non-decreasing and constant for $r\ge R$,  (ii) is trivially true in the cases $\tilde r \le r$ and $r \ge R$.  In the case $r<\min (\tilde r,R)$,  \begin{align*}
f(\tilde r) - f(r) &\ =\ 
\frac 1a\left( e^{-ar}-e^{-a\min (\tilde r,R)}\right)
\  \le\ \frac 1a f_{-}^{\prime}(r)  \;.
\end{align*}  
Combining these cases gives (ii).  Let  \[
g(x) := \frac{x}{e^{x} - 1} ~~ \text{so that $\frac{r  f^{\prime}_{-}(r) }{f(r)} =  g(ar)$} \;.
\] Property (iii) then follows because $g$ decreases with $x$,  $\lim_{x \to 0} g(x) = 1$ and $g(x) \ge \max(1,x) e^{-\max(1,x)}$.
\end{proof}

\section{Proof of Foster-Lyapunov drift condition}  \label{sec:LYAPproof}

Before giving the proof of Lemma \ref{lem:LYAP}, here are some preparatory results.   Using the shorthand notation $\varrho(t)=\snorm{q_t(x,\xi )}^2$ and $\varphi(t) = \snorm{v_t(x,\xi )}^2$, \eqref{eq:dynamicsgeneral} implies 
\begin{eqnarray*}
\varrho'(t) &=& 2 \langle q_t(x,\xi ), v_t(x,\xi ) \rangle_s \;, \\
\varrho''(t) &=& 2 \left( \varphi(t) + \langle q_t(x,\xi ), b(q_t(x,\xi )) \rangle_s \right) \;, \\
\varphi'(t)&=& 2 \langle v_t(x,\xi ), b(q_t(x,\xi )) \rangle_s  \;.
\end{eqnarray*}
Hence, by Assumption \ref{B123} (B1) and (B3), we have the differential inequalities \begin{align}
\varrho''(t) &\le 2 \left( \varphi(t) - K\varrho(t) + A \right)\;, \qquad - \varrho''(t) \le 2 \L \varrho(t) \;.  \label{eq:q2} 
\end{align}
The following formula comes from two applications of integration by parts and is valid for any $k \in \mathbb{N}$ and for any twice differentiable function $g: \mathbb{R} \to \mathbb{R}$,  \begin{equation} \label{eq:byparts}
\int_0^T g(r) (T-r)^k dr =  \frac{g(0) T^{k+1}}{k+1} +  \frac{g'(0) T^{k+2}}{(k+1) (k+2)} + \int_0^T   \frac{g''(r) (T-r)^{k+2}}{(k+1) (k+2)} dr \;.
\end{equation} 
We also require the following inequalities 
\begin{eqnarray}
     \sup_{r \le T} \varrho(r)  &\le & 4 (\snorm{x} + T \snorm{\xi })^2 
     \ \le\ 8 ( \varrho(0) + T^2 \varphi(0) ) \;, \label{eq:AP1}  \\
    \sup_{r \le T} \varphi(r) &\le & 8 \L^2 T^2 \varrho(0) + 2 (1+2\L T^2)^2 \varphi(0) \ \le\  8 \L^2 T^2 \varrho(0) + \frac{9}{2} \varphi(0) ,\ \ \label{eq:AP2}  
\end{eqnarray}
which follow from Lemma~\ref{0A} and the assumption $\L T^2 \le 1/48$.

\begin{proof}[Proof of Lemma \ref{lem:LYAP}]
Apply in turn \eqref{eq:q2}, then \eqref{eq:byparts} with $g(r) = \varrho(r)$, and then \eqref{eq:q2} again, to obtain 
\begin{eqnarray*}
\lefteqn{E  \left[ \snorm{X'(x)}^2 \right] \ =\ E  \left[ \rho (T) \right]
\ =\
\rho (0)+T\, E[\rho'(0)]+
 \int_0^T (T-r) E \left[ \varrho''(r) \right] dr } \\
&\le & \snorm{x}^2 + 2 \int_0^T (T-r)\,  E \left[ \varphi(r) - K\varrho(r) + A \right] dr \nonumber  \\
&\le &(1 - K T^2) \snorm{x}^2 +  A T^2  
+ 2 \int_0^T (T-r) E \left[  \varphi(r) \right] dr \\&&\quad - 2 K \int_0^T \frac{(T-r)^3}{3!} E \left[   \varrho''(r) \right] dr  \\
&\le & (1 - K T^2) \snorm{x}^2 +   A T^2  
+ T^2  E \left[ \sup_{r \le T} \varphi(r) \right]   + \frac 16 K LT^4  \, E \left[   \sup_{r \le T} \varrho(r) \right]   \\
&\le &\left(1 - K T^2 + 8 L^2 T^4 + \frac{4}{3} K L T^4 \right) \snorm{x}^2+ A T^2  +   \left( \frac{9}{2}  + \frac{4}{3} K \L T^4 \right) T^2   E\snorm{\xi }^2 
\end{eqnarray*} 
where in the last step we applied \eqref{eq:AP1} and \eqref{eq:AP2}.
Since we assume $\L T^2 \le (1/48) ( K/\L )$, note that $8 \L^2 T^4 \le (1/6) K T^2$, and since also $KT^2\le \L T^2 \le 1/4$,  \begin{align*}
E \left[ \snorm{X'(x)}^2 \right]  \le \left(1 - {K T^2}/{2}  \right) \snorm{x}^2 + 5 (\tr(\tilde{\mathcal{C}} \mathcal{C}^{-s}) + A) T^2 
\end{align*}
{as required.}
\end{proof}

\section{Proofs of main results}  \label{sec:CONTRproof}

\begin{proof}[Proof of Theorem \ref{thm:CONTR}]
The parameters $\gamma$, $a$ and $\epsilon$ have been chosen in \eqref{Cgamma}, \eqref{Ca}, and \eqref{Ce} such that the following conditions are satisfied:
\begin{eqnarray}
\label{eq:A}\gamma T &\le &1,\\
\label{eq:B}\gamma R &\le &1/4,\\
\label{eq:C} a T &= &1, \\
\label{eq:D}  (\sigma_{max}/{\sigma_{min}} ) \L T &\le & {\gamma}/{4} , \\
\label{eq:E} \epsilon (A+\tr( \tilde{\mathcal{C}} \mathcal{C}^{-s} ) )  &= & \ \max (1,R/T)\,e^{ - \max (1,R  / T)}/160 \;.
\end{eqnarray}
Indeed, \eqref{eq:A} and \eqref{eq:B} hold by selection of $\gamma$  in \eqref{Cgamma}; 
\eqref{eq:C} holds by selection of $a$ in \eqref{Ca}; \eqref{eq:D} holds  because \eqref{A0A} implies that 
\begin{equation}\label{boundTR}
  ( \sigma_{max}/\sigma_{min} ) T \le \min(T^{-1}/4, R^{-1}/16) =\gamma/4  
\end{equation}
  by selection of $\gamma$ in \eqref{Cgamma}; and \eqref{eq:E} holds by selection of $\epsilon$ in \eqref{Ce}.

\medskip

Let $z=x-y$, $W=\xi-\eta$,  $r=\anorm{z}$, $R'= \anorm{X'(x,y)-Y'(x,y)}$, $G=1+\epsilon (\snorm{x}^2+\snorm{y}^2)$, $G'=1+\epsilon (\snorm{X'}^2+\snorm{Y'}^2) $, $F=f(r)$ and $F'=f(R')$.  We consider separately the cases where $r < R$ and $r \ge R$.\medskip

\medskip

{\em (i) Contractivity for $r < R$.} \ Expand \begin{align} \label{eq:tdfs}
E [ F' - F ] =  \rn{I} + \rn{II} \;, \quad \text{where} \quad \begin{cases} \rn{I} =E [ F' - F ; ~W^{\ell} = -\gamma z^{\ell}] \;, \\ \rn{II} =  E[ F' - F; ~W^{\ell} \ne -\gamma z^{\ell} ] \;.\end{cases} 
\end{align}
Let $Z_T = q_T(x,\xi) - q_T(y, \eta)$.  By Lemma~\ref{lem:f} (i),  Lemma~\ref{lem:contr}, \eqref{eq:D} and \eqref{Calpha},
\begin{align}
 \rn{I} &\ \le\ f'(r) E [ R' - r ; ~ W^{\ell} = -\gamma z^{\ell}] \nonumber \\
  &\ =\  f'(r) E \left[ \snorm{\tilde{\mathcal{C}}^{-\frac{1}{2}} \mathcal{C}^{\frac{s}{2}} Z^{\ell}_T} - \snorm{\tilde{\mathcal{C}}^{-\frac{1}{2}} \mathcal{C}^{\frac{s}{2}} z^{\ell}} + \alpha ( \snorm{Z^h_T} - \snorm{z^h} ) ;~ W^{\ell} = -\gamma z^{\ell} \right]  \nonumber \\
 &\ \le\  f'(r) \left( ( -\gamma T + \frac{5}{8} \frac{\sigma_{max}}{\sigma_{min}}  \L T^2 + \frac{1}{4}  \sigma_{min}^{-1} \alpha T^2 ) \snorm{\tilde{\mathcal{C}}^{-\frac{1}{2}} \mathcal{C}^{\frac{s}{2}} z^{\ell}}  \right. \nonumber \\
& \qquad\qquad\qquad  \left. +\, ( - \frac{1}{4} \alpha T^2  + \frac{5}{8} \sigma_{max} \L T^2 ) \snorm{z^h}    \right) ( 1 - P[ W^{\ell} \ne -\gamma z^{\ell} ])\nonumber  \\
 &\ \le\  -f'(r) \left( \frac{19}{32}\gamma T \snorm{\tilde{\mathcal{C}}^{-\frac{1}{2}} \mathcal{C}^{\frac{s}{2}} z^{\ell}} \, +\,  \frac{3}{32} \alpha T^2  \snorm{z^h}    \right) ( 1 - P[ W^{\ell} \ne -\gamma z^{\ell} ])  \label{eq:dfsI}
\end{align}
Moreover, by Lemma~\ref{lem:X} and Lemma~\ref{lem:f} (ii),
\begin{align}
 P[ W^{\ell} \ne -\gamma z^{\ell} ] &\ \le\ \frac{\gamma R}{\sqrt{2 \pi}}\ \le\ \frac{1}{4} \frac{1}{ \sqrt{2 \pi}}\ <\  \frac{1}{10} \;,\qquad\text{ and} \label{eq:dfsIP} \\
\rn{II} &\ \le\ a^{-1} f'(r) P[ w^{\ell} \ne -\gamma z^{\ell} ]\ \le\   f'(r) \frac{2}{5} \gamma T  \snorm{\tilde{\mathcal{C}}^{-\frac{1}{2}} \mathcal{C}^{\frac{s}{2}} z^{\ell}}  \;,  \label{eq:dfsII}
\end{align}
where in \eqref{eq:dfsIP} we used \eqref{eq:B} and in \eqref{eq:dfsII} we used \eqref{eq:C} and $\sqrt{2\pi }>5/2$.\smallskip\\
Inserting \eqref{eq:dfsI}, \eqref{eq:dfsIP}, and \eqref{eq:dfsII} into \eqref{eq:tdfs}, and using Lemma~\ref{lem:f} (iii), gives 
\begin{eqnarray}
\nonumber
\lefteqn{ E [ F' - F ] \ \le \ f'(r) \left( 
 - \tfrac{1}{8}  \gamma T \snorm{\tilde{\mathcal{C}}^{-\frac{1}{2}} \mathcal{C}^{\frac{s}{2}} z^{\ell}}  - \tfrac{1}{12} \alpha T^2   \snorm{z^h}    \right)} \\
 &\le & -\tfrac 1{12}\, \min (\gamma T,T^2)\, rf'(r)\label{eq:dfs}  \\
\nonumber
&\le &  - \tfrac{1}{12} \, \min (1, T/(4R),T^2)\,\max (1,R/T)\,e^{-\max (1,R/T)}\, f(r)\\
&\le & -c_1 \, F.\nonumber
\end{eqnarray}
Here we have introduced $c_1  := (1/12) T^2 \max (1,R/T) e^{-\max (1,  R / T) }$, and we have used \eqref{Cgamma}, \eqref{eq:C}, and 
the fact that $T^2\le \min (1,T/(4R))$ by \eqref{boundTR}.\smallskip\\
Furthermore, by Lemma~\ref{lem:LYAP}, 
\begin{align}
E [ G' - G ]\ &\le\  10 \epsilon (A+\tr(\tilde{\mathcal{C}} \mathcal{C}^{-s})) T^2 \
\le\ (3/4) {c_1} G \;, \label{eq:dGs}
\end{align} 
where in the last step we eliminated $\epsilon$ using \eqref{eq:E}.\smallskip\\
The Cauchy-Schwarz inequality,  \eqref{eq:dfs} and \eqref{eq:dGs} now imply 
\begin{eqnarray}
\lefteqn{E[ \rho(X',Y')]\ =\ E[ \sqrt{F' G'} ]\ \le\ \sqrt{ E[F'] } \sqrt{ E[G'] }} \label{eq:CS}\\
& \le & \sqrt{(1-c_1)F}\sqrt{ (1 +  3{c_1}/{4} )G}\  \le\ \sqrt{ 1 - {c_1}/{4}} \sqrt{FG}  \nonumber\\
&\le & \exp\left( - {c_1}/{8} \right) \rho(x,y), \label{eq:drho1}
\end{eqnarray}
where in the last step we used $1-\mathsf{x} \le e^{-\mathsf{x}}$ with $\mathsf{x}=c_1/4$.

\medskip

{\em (ii) Contractivity for $r \ge R$.} \ In this case, by \eqref{ieq:R} and \eqref{compnormB}, we have $\snorm{x}^2 + \snorm{y}^2 \ge 40  (A+\tr(\tilde{\mathcal{C}} \mathcal{C}^{-s})) / K$, and we can apply the Foster-Lyapunov drift condition in \eqref{eq:FL_coupling} to \eqref{eq:CS} to obtain \begin{align}
E &[ \rho(X',Y')] \le \sqrt{F} \sqrt{ E[G'] }  \le \sqrt{F} \left(1 + \epsilon \left( 1 - \frac{K T^2}{4} \right) (\snorm{x}^2 + \snorm{y}^2) \right)^{1/2}   \nonumber \\
&\le \sqrt{F} \left(    1-  \frac{5}{2} \epsilon  (A+\tr(\tilde{\mathcal{C}} \mathcal{C}^{-s})) T^2 + \epsilon \left( 1 - \frac{K T^2}{8} \right)  (\snorm{x}^2 + \snorm{y}^2) \right)^{1/2}  \nonumber \\
&\le \left(1 - c_2 \right)^{1/2} \rho(x,y)\, \le\, \exp\left( - {c_2}/{2} \right) \rho(x,y) \label{eq:drho2}
\end{align}
where $c_2:=  \min\left( K T^2/8,\,   T^2 \max (1,R/T)e^{-\max (1,R/T)}
/64 \right)$.

\medskip

{\em (iii) Global Contraction.} 
Let $c:=\min(c_1/8,c_2/2)=c_2/2$.  By combining the bounds in \eqref{eq:drho1} and \eqref{eq:drho2}, we see that for any $x,y \in \H^s$, \[
E [ \rho(X',Y')] \le e^{-c} \rho(x,y) \;.
\]
\end{proof}

\begin{proof}[Proof of Corollary \ref{cor:QBHMC}]
The Wasserstein contraction in \eqref{QBpHMC1} follows directly from Theorem \ref{thm:CONTR}, see e.g.\ \cite[Corollary 2.8]{BoEbZi2018} for a similar result.
The bound in \eqref{QBpHMC3} then follows from \eqref{QBpHMC1} by comparing
$\rho$ to the metric on $\H^s$. Indeed, recall that by \eqref{eq:f}
and \eqref{Ca}, $f(r)=(1-e^{-\min (r,R)/T}) T$. Let $x,y\in\H^s$, and let 
$r=\anorm{x-y}$. Suppose first that $r\le T$. Then $f(r)\ge r/e$, and thus 
$$\rho (x,y)\ \ge\ \sqrt{f(r)}\ \ge\ \sqrt{r/e}\ \ge\ r/\sqrt{eT}\ \ge
\ \snorm{x-y}\min(\sigma_{min},\alpha)/\sqrt{eT}$$
where in the last step we used \eqref{compnormA}.
Now suppose that $r>T$. Then since also $R\ge T$ by assumption,
$f(r)\ge (1-e^{-1})T$, and thus
we obtain
$$\rho (x,y)\ \ge\ \sqrt{(1-e^{-1})T\epsilon}\sqrt{\snorm{x}^2+\snorm{y}^2}\ \ge\ \sqrt{(1-e^{-1})T{\epsilon}/2}\snorm{x-y}.$$
Combining both cases and noting that by \eqref{Calpha}, $\alpha\ge \sigma_{min}$, we see that
$$\snorm{x-y}\ \le\ \max\left(
\sqrt{eT}/\sigma_{min}, \sqrt 2/\sqrt{(1-e^{-1})T{\epsilon}}
\right)\, \rho (x,y)\quad\text{for all }x,y 
,$$
which implies an analogue bound for the corresponding Wasserstein distances $\mathcal W^{s,1}$ and $\mathcal W_\rho$. Conversely, 
since $f(r)\le T$ for all $r$,
$$\rho (x,y)\ \le\ \sqrt{T}\sqrt{1+\epsilon\snorm{x}^2+\epsilon\snorm{y}^2}\ \le\    \sqrt{T} (1+\sqrt\epsilon\snorm x+\sqrt\epsilon\snorm{y})\quad\text{for all }x,y .     $$
Therefore, with $C$ defined by \eqref{Cdefi}, we obtain
\begin{eqnarray*}
\lefteqn{\mathcal W^{s,1} (\nu\pi^k ,\mu )\ = \ \mathcal W^{s,1} (\nu\pi^k ,\mu\pi^k )\ \le\ CT^{-1/2}\,
\mathcal W_\rho (\nu\pi^k,\mu\pi^k) }\\
& \le & CT^{-1/2}e^{-c k}\, \mathcal W_\rho (\nu ,\mu )
\ \le \ C\left( 1+\sqrt\epsilon M_1(\nu )+\sqrt\epsilon M_1(\mu )\right)e^{-c k}
\end{eqnarray*}
for all $k\in\mathbb N$ and all probability measures $\nu$ on $\H^s$.  Finally, by Lemma~\ref{lem:LYAP} and \eqref{Ce}, we have $\sqrt\epsilon M_1(\mu ) \le (1/4) K^{-1/2} e^{-R/(2 T)}$.
\end{proof}

\section{Proofs of results from Section 3 (Applications)} 

\subsection{Proofs of results for TPS}
 \label{sec:TPS_proofs}

To prove Theorem~\ref{thm:CONTR_TPS}, we compare the eigenvalues of $\bph{\mathcal{C}}$ to $\mathcal{C}$.  Note that these eigenvalues each have multiplicity $d$, and to account for this,  define the index function $\varphi(k,j) =  d (k-1) +j$.  Then the eigenvalues of  $\mathcal{C}$ are  \begin{equation} \label{lambda_tps}
\lambda_{\varphi(k,j)}= \left( \frac{ \tau }{k \pi} \right)^{2} \;, \quad k \in \mathbb{N} \;, \quad 1 \le j \le d \;,
\end{equation} and the eigenvalues of $\bph{\mathcal{C}}$ are 
\begin{equation} \label{Lambda_tps}
\bph{\Lambda}_{\varphi(k,j)} = \lambda_{\varphi(k,j)} \left( \frac{\theta_k}{\sin(\theta_k)} \right)^2, ~ \theta_k:=\frac{k \pi}{2 (m+1)},  ~ 1 \le k \le m, ~ 1 \le j \le d.
\end{equation}   
The following lemma helps estimate the error of the eigenvalues of the approximation $\bph{\mathcal{C}}$ relative to those of $\mathcal{C}$. 

\begin{lemma} \label{eigs_TPS}
For any $m \in \mathbb{N}$, for all $1 \le k \le m$, and for $1 \le j \le d$,
\begin{description}
    \item[(E1)] $| \bph{\Lambda}_{\varphi(k,j)} - \lambda_{\varphi(k,j)} | = \bph{\Lambda}_{\varphi(k,j)} - \lambda_{\varphi(k,j)} \le \lambda_{\varphi(k,j)} \dfrac{k^2 \pi^2 }{6 (m+1)^2} = \dfrac{\tau^2}{6 (m+1)^2}$, 
    \item[(E2)] 
    %$\left(\dfrac{\lambda_1}{\lambda_{\varphi(k,j)}}\right)^{1/2}-\dfrac{\pi^2 k^2}{12 (m+1)^2} \le 
    $\left(\dfrac{ \bph{\Lambda}_1}{ \bph{\Lambda}_{\varphi(k,j)}}\right)^{1/2} \le \left(\dfrac{\lambda_1}{\lambda_{\varphi(k,j)}}\right)^{1/2} \left( 1  + \dfrac{\pi^2}{16 (m+1)^2} \right)$.
\end{description}
\end{lemma}

\begin{proof}
This lemma is an easy consequence of the elementary inequalities \begin{equation} \label{eigs_ieqs}
\frac{1}{2} < 1 - \frac{\theta^2}{6} < \frac{\sin(\theta)}{\theta} < 1 \quad \text{and} \quad 1 < \frac{\theta}{\sin(\theta)} < 1 + \frac{\theta^2}{4} < \frac{5}{3}
\end{equation} which are valid for $0 < \theta < \pi/2$.   Indeed, \eqref{lambda_tps}, \eqref{Lambda_tps} and \eqref{eigs_ieqs} imply  \begin{align*}
| \bph{\Lambda}_{\varphi(k,j)}& - \lambda_{\varphi(k,j)} | = \bph{\Lambda}_{\varphi(k,j)} - \lambda_{\varphi(k,j)} \\
&= \lambda_{\varphi(k,j)} \left( \left( \frac{\theta_k}{\sin(\theta_k)} \right)^2  - 1 \right) \\
&= \lambda_{\varphi(k,j)}  \left(  \frac{\theta_k}{\sin(\theta_k)}   + 1 \right) \left(  \frac{\theta_k}{\sin(\theta_k)}   - 1 \right) \\
&\le  \lambda_{\varphi(k,j)}\frac{2}{3} \theta_k^2  =  \lambda_{\varphi(k,j)} \frac{k^2 \pi^2}{6 (m+1)^2} = \frac{\tau^2}{6 (m+1)^2}
\end{align*}
as required for (E1).  For (E2), we use \eqref{Lambda_tps} to write \[
\left(\dfrac{ \bph{\Lambda}_1}{ \bph{\Lambda}_{\varphi(k,j)}}\right)^{1/2} = \left(\dfrac{\lambda_1}{\lambda_{\varphi(k,j)}}\right)^{1/2}
 \frac{\sin(\theta_k)}{\theta_k}  \frac{\theta_1}{\sin(\theta_1)} \;.
\] Hence, by \eqref{eigs_ieqs}, \[
%k \left( 1- \frac{\theta_k^2}{6} \right) \le 
\left(\dfrac{ \bph{\Lambda}_1}{ \bph{\Lambda}_{\varphi(k,j)}}\right)^{1/2} \le \left(\dfrac{\lambda_1}{\lambda_{\varphi(k,j)}}\right)^{1/2}  \left( 1 + \frac{\theta_1^2}{4} \right) = \left(\dfrac{\lambda_1}{\lambda_{\varphi(k,j)}}\right)^{1/2}  \left(1 + \frac{\pi^2}{16 (m+1)^2} \right) 
\] as required for (E2).
\end{proof}

\begin{proof}[Proof of Theorem~\ref{thm:CONTR_TPS}]
This result is an application of Corollary~\ref{cor:QBHMC}.  Since $\bph{\mathcal{C}}$ is a finite-dimensional matrix, Assumption~\ref{A456} holds for $\bph{\mathcal{C}}$ with $s=0$, and since we choose $\bph{\tilde{\mathcal{C}}} = \bph{\mathcal{C}}$, Assumption~\ref{C123} also holds with $s=0$.  Therefore, to apply Corollary~\ref{cor:QBHMC}, it suffices to check that: (i) Assumption~\ref{B123} holds with dimension-free constants $L$, $n$, $K$, and $A$, (ii) the dimension-free $R$ defined in \eqref{R_TPS} satisfies condition~\eqref{ieq:R}, and (iii) the dimension-free condition~\eqref{A0A_TPS} on the duration $T$ implies \eqref{A0A} holds.  We then invoke Corollary~\ref{cor:QBHMC} to conclude convergence in the standard $L^1$ Wasserstein distance.

\medskip

{\em Verify Assumption~\ref{B123} (B1)-(B3).} For (B1), note that \begin{align*}
|\bph{b}(\bph{x}) - \bph{b}(\bph{y}) | &\le (1+\bph{\Lambda}_1 L_G) |\bph{x}-\bph{y}| \\
&\le (1+2 \lambda_1 L_G)  |\bph{x}-\bph{y}|  = (1+2 (\tau^2 / \pi^2) L_G)  |\bph{x}-\bph{y}|
\end{align*}
where in the last step we used Lemma~\ref{eigs_TPS} (E1) which implies that $\bph{\Lambda}_1 - \lambda_1 \le \lambda_1 \pi^2 / (6 (m+1)^2) \le \lambda_1$ since $m \ge 1$.  Thus, (B1) holds with $L=1+\kappa$ since $\kappa = 2 (\tau^2 / \pi^2) L_G$. For (B2), since $n=m_{\ell} d = \varphi(m_{\ell},d)$, $n+1=\varphi(m_{\ell}+1,1)$ and \begin{align*}
    |\bph{b}^h(\bph{x}) + \bph{x}^h &- \bph{b}^h(\bph{y}) - \bph{y}^h | 
    \le \bph{\Lambda}_{\varphi(m_{\ell}+1,1)} L_G  |\bph{x}-\bph{y}| \\
    &\le \lambda_{\varphi(m_{\ell}+1,1)} L_G |\bph{x}-\bph{y}| + ( \bph{\Lambda}_{\varphi(m_{\ell}+1,1)} - \lambda_{\varphi(m_{\ell}+1,1)}) L_G |\bph{x}-\bph{y}| \\
    &\le 2  \lambda_{\varphi(m_{\ell}+1,1)} L_G |\bph{x}-\bph{y}| \le (1/3) |\bph{x}-\bph{y}| 
\end{align*}
where in the second to last step we used \[ \bph{\Lambda}_{\varphi(m_{\ell}+1,1)} - \lambda_{\varphi(m_{\ell}+1,1)} \le \lambda_{\varphi(m_{\ell}+1,1)} (m_{\ell}+1)^2 \pi^2 / (6 (m+1)^2) \le \lambda_{\varphi(m_{\ell}+1,1)}
\] which follows from Lemma~\ref{eigs_TPS} (E1) since $m \ge (m_{\ell}+1) \pi/2$, and in the last step, we used that $m_{\ell}+1 \ge  \sqrt{6 L_G} \tau/\pi$.  Hence, (B2) holds with $n=m_{\ell} d$.
For (B3), \begin{align*}
    \langle \bph{x}, \bph{b}(\bph{x}) \rangle 
    &\le  - | \bph{x}|^2  +  | \bph{x}|  |\bph{\mathcal{C}} \nabla G_{m}(\bph{x})| \le  - (1/2) | \bph{x}|^2  + (1/2)  |\bph{\mathcal{C}} \nabla G_{m}(\bph{x})|^2 \\
    &\le  - (1/2) | \bph{x}|^2  + (1/2) \bph{\Lambda}_1^2 M_G^2 \tau \le  - (1/2) | \bph{x}|^2  + \lambda_1^2 \tau M_G^2 
\end{align*}
where in the last step we used \[
\bph{\Lambda}_1^2 - \lambda_1^2 = 2 \lambda_1 (\bph{\Lambda}_1 - \lambda_1) + (\bph{\Lambda}_1 - \lambda_1)^2 \le \lambda_1^2
\] which follows from (E1) since $m \ge 1$.  Thus, (B3) holds with $K=1/2$ and $A= \lambda_1^2 \tau M_G^2  = ( \tau^5 / \pi^4) M_G^2$.  
To summarize, Assumption~\ref{B123} holds with dimension-independent constants $L=1+\kappa$, $n=m_{\ell} d$ where $m_{\ell} = \lfloor \sqrt{3 \kappa} \rfloor$, $K=1/2$, and $A= (\tau^5 / \pi^4) M_G^2$.

\medskip

{\em Verify Conditions~\eqref{ieq:R} \& \eqref{A0A}.} To show that $R$ defined in \eqref{R_TPS} satisfies condition~\eqref{ieq:R} and that condition~\eqref{A0A_TPS} on the duration parameter implies condition~\eqref{A0A}, in this paragraph we gather some additional bounds.  Since $m_{\ell} \le \sqrt{3 \kappa}$, we have
\begin{equation} \label{sigma_max_tps}
\sigma_{max} = \bph{\Lambda}_{\varphi(m_{\ell},1)}^{-1/2} = \lambda_{\varphi(m_{\ell},1)}^{-1/2} (\sin(\theta_{m_{\ell}})/\theta_{m_{\ell}}) \le   m_{\ell} \pi / \tau \le \sqrt{6 L_G} \;,
\end{equation} \begin{equation}
\label{sigma_min_tps}
\sigma_{min}^{-1} = \bph{\Lambda}_1^{1/2} = \lambda_1^{1/2} + \lambda_1^{1/2} \left( \theta_1/\sin(\theta_1) - 1 \right) \le 2 \lambda_1^{1/2} = 2 \tau/\pi\;,
\end{equation}
where in \eqref{sigma_max_tps} and \eqref{sigma_min_tps} we used \eqref{eigs_ieqs}.  Moreover, by Lemma~\ref{eigs_TPS} (E2),  \begin{equation}
\label{sigma_max_min}
\dfrac{\sigma_{max}}{\sigma_{min}} = \left(\dfrac{ \bph{\Lambda}_1}{ \bph{\Lambda}_{\varphi(m_{\ell},1)}}\right)^{1/2} \le \left(\dfrac{ \lambda_1}{ \lambda_{\varphi(m_{\ell},1)}}\right)^{1/2} \left( 1  +  \frac{\pi^2}{16 (m+1)^2} \right) \le 2 m_{\ell}
\end{equation} since $m \ge 1$ and $\sqrt{\lambda_1/\lambda_{\varphi(m_{\ell},1)}} = m_{\ell}$, and by Lemma~\ref{eigs_TPS} (E1), \begin{align}
\tr(\bph{\mathcal{C}}) &\le \tr(\mathcal{C})+ \sum_{i=1}^{m d} ( \bph{\Lambda}_i -  \lambda_i) \le  \tr(\mathcal{C})+ (d \tau^2 / 6) \ m/(m+1)^2 \nonumber \\
&\le  2 \tr(\mathcal{C}) = d \tau^2 / 3 \;, \quad \text{since $\tr(\mathcal{C}) = d \tau^2 / 6$.}
\label{trace_C_TPS}
\end{align}   

Let $R_m=8 \sqrt{40} (A + \tr(\bph{\mathcal{C}}))^{1/2} \sigma_{max} L K^{-1/2}$ denote the RHS of \eqref{ieq:R}.
Then using \eqref{sigma_max_tps}, \eqref{trace_C_TPS}, $L=1+\kappa$, $K=1/2$, and $A=M_G^2 \tau^5 / \pi^4$, we have \begin{align*}
R_m^2 &\le 128 \times 40 \times 6 L_G (1+\kappa)^2 ( (\tau^5 / \pi^4)  M_G^2 + \tr(\bph{\mathcal{C}}) ) \\
&\le 256 \times 20 \times \kappa (1 + \kappa)^2 \left( 3  (\tau^3/\pi^2) M_G^2 + d \pi^2 \right)  \\
&\le 256 \times 20 \pi^2 \kappa (1+\kappa)^2 ( (\tau/ \pi)^3 M_G^2 + d) = R^2
\end{align*} which implies that $R$ defined in \eqref{R_TPS} satisfies \eqref{ieq:R}.
Moreover, by \eqref{sigma_max_min},  \begin{equation} \label{smmL_A0A_TPS}
\dfrac{\sigma_{max}}{\sigma_{min}} \L \le 2 m_{\ell} (1+\kappa)  \le 2 \sqrt{3 \kappa} (1+\kappa)  \;.
\end{equation} 
Inserting \eqref{smmL_A0A_TPS} into the LHS and RHS of \eqref{A0A} gives \eqref{A0A_TPS}.  Thus, whenever $T$ satisfies condition~\eqref{A0A_TPS} then condition~\eqref{A0A} holds.

\medskip

{\em Invoke Corollary~\ref{cor:QBHMC}.}
By Corollary~\ref{cor:QBHMC} and using $K=1/2$, as long as $T$ satisfies \eqref{A0A_TPS}, \begin{align}  \label{QBpHMC_TPS_1}
\mathcal W^{0,1}  (\nu_m \pi_m^k, \mu_m) &\le  C_{m} e^{-c k} \left( 1 + \sqrt{\epsilon_m} M_1(\nu_m) + (1/8) e^{-R/(2 T)} \right) 
\end{align} holds with
the dimension-free rate in \eqref{crate_TPS} and the constants: \begin{equation} \label{QBpHMC_TPS_Cm}
\begin{aligned}
C_{m} &= \max(2 T \sigma_{min}^{-1}, 23 (A + \tr(\bph{\mathcal{C}}))^{1/2} e^{R/(2 T)}) \;, \\
\epsilon_m &= (1/160) (A + \tr( \bph{\mathcal{C}} ))^{-1} e^{-R/T} \;.
\end{aligned}
\end{equation}
These dimension-dependent constants can be upper bounded by dimension-free constants $C$ and $\epsilon$ given in \eqref{C_TPS} and \eqref{eps_TPS}, by using $A=(\tau^5 / \pi^4) M_G^2 $, \eqref{sigma_min_tps} and \eqref{trace_C_TPS}. 
Thus, \eqref{QBpHMC_TPS} holds.
\end{proof}

\begin{comment}
where $U: \H \to \mathbb{R}$ is defined as: \[
U(x) = \int_0^S G(x_s) ds \;, 
\] and $G: \mathbb{R} \to \mathbb{R}$ is the potential energy associated with a mixture of two Gaussians with means $\pm 2 \sigma$ and variances $\sigma^2$ where $\sigma >0$, i.e.,   \[
G(x_s)=- \log\left( \exp\left(-\dfrac{(x_s-2 \sigma)^2}{2 \sigma^2} \right) + \exp\left(-\dfrac{(x_s+2 \sigma)^2}{2 \sigma^2} \right) \right) \;.
\]  In this case, $\sigma^2 G''(x_s) = {1 - 4 \operatorname{sech}\left( {2 x_s}/{\sigma} \right)^2 } $ and $\, L_g = \sup |G''| = 3 / \sigma^2$ 
%and $G''(x_s) \ge 2/(3 \sigma^2)$ for all $|x_s| > \sigma$.   
\end{comment}

\subsection{Proofs of results for PIMD}
 \label{sec:PIMD_proofs}

To prove Theorem~\ref{thm:CONTR_PIMD}, we compare the eigenvalues of $\bph{\mathcal{C}}_{\mathsf{a}}$ to $\mathcal{C}_{\mathsf{a}}$. The leading eigenvalue of $\mathcal{C}_{\mathsf{a}}$ has multiplicity $d$, while all of the other eigenvalues have multiplicity $2d$.  If $m$ is odd, then the leading eigenvalue of both $\bph{\mathcal{C}}_{\mathsf{a}}$ has multiplicity $d$, while all of the other eigenvalues of $\bph{\mathcal{C}}_{\mathsf{a}}$ have multiplicity $2 d$. However, if $m$ is even, then the trailing and leading eigenvalues of $\bph{\mathcal{C}}_{\mathsf{a}}$ have multiplicity $d$, while all of the other eigenvalues of of $\bph{\mathcal{C}}_{\mathsf{a}}$ have multiplicity $2 d$. To account for these multiplicities, it is helpful to define the index function \[
\varphi(k,j) = \begin{cases}   j & \text{if $k=1$,~~  $1 \le j \le d$} \;, \\
2 d (k-2)+j+d  & \text{if $k>1$, ~~ $1 \le j \le 2 d$} \;.
\end{cases}
\]  
For any $k\in \mathbb{N}$, the eigenvalues of $\mathcal{C}_{\mathsf{a}}$ are  \begin{equation} \label{lambda_pimd}
\lambda_{\varphi(k,j)}= \begin{cases} \mathsf{a}^{-1} & \text{if $k=1$, ~~ $1 \le j \le d$}, \\ \dfrac{1}{\mathsf{a}+\omega_k^2} &  \text{if $k>1$,~~ $1 \le j \le 2 d$}.
\end{cases}
\end{equation}  For all $m \in \mathbb{N}$ and $k \in \{1, \dots, \lceil (m+1)/2 \rceil \}$,
the eigenvalues of $\bph{\mathcal{C}}_{\mathsf{a}}$ are 
\begin{equation} \label{Lambda_pimd}
\bph{\Lambda}_{\varphi(k,j)} = \begin{cases} 
\mathsf{a}^{-1}  & \text{if $k=1$, ~~ $1 \le j \le d$,} \\
\dfrac{1}{\mathsf{a} +\omega_k^2 \sin^2(\theta_k)/\theta_k^2}  &   \begin{cases}  
\text{if $k>1$ \& $k \ne \frac{m}{2}+1$}, & 1 \le j \le 2 d,   \\
\text{if $k=\frac{m}{2}+1$}, & 1 \le j \le d,
\end{cases}
\end{cases}
\end{equation} 
 Here we have introduced 
\begin{equation} \label{eq:om2_k}
\theta_k := \
\frac{(k-1) \pi}{m} \quad \text{and} \quad
\omega_{k}^2 :=  \frac{4  (k-1)^2 \pi^2}{\beta^2} \;.
\end{equation}
Note that the definition in \eqref{Lambda_pimd} includes  odd or even values of $m$.  The following lemma estimates the error of the eigenvalues of $\bph{\mathcal{C}}_{\mathsf{a}}$ relative to those of $\mathcal{C}_{\mathsf{a}}$.

\begin{lemma} \label{eigs_PIMD}
For any $m \in \mathbb{N}$ and $k \in \{1, \dots, \lceil (m+1)/2 \rceil \}$,
\begin{description}
    \item[(E1)] $| \bph{\Lambda}_{\varphi(k,1)} - \lambda_{\varphi(k,1)} | = \bph{\Lambda}_{\varphi(k,1)} - \lambda_{\varphi(k,1)} \le \lambda_{\varphi(k,1)} 2 \theta_k^2 $ \;.
    \item[(E2)] 
    %$\left(1-\dfrac{\theta_k^2}{6} \right)^{1/2} \left(\dfrac{\lambda_1}{\lambda_{\varphi(k,j)}}\right)^{1/2} \le 
    $\left(\dfrac{ \bph{\Lambda}_1}{ \bph{\Lambda}_{\varphi(k,1)}}\right)^{1/2} \le \left(\dfrac{\lambda_1}{\lambda_{\varphi(k,1)}}\right)^{1/2}$ \;.
\end{description}
\end{lemma}

\begin{proof}
This lemma is an easy consequence of the inequalities in \eqref{eigs_ieqs}.  For $k=1$, (E1) and (E2) trivially hold since $\bph{\Lambda}_{\varphi(k,1)}=\lambda_{\varphi(k,1)}=\mathsf{a}^{-1}$.  For $k>1$, \eqref{Lambda_pimd}, \eqref{lambda_pimd}, and \eqref{eigs_ieqs} imply  \begin{align*}
| \bph{\Lambda}_{\varphi(k,j)} &- \lambda_{\varphi(k,j)} | = \bph{\Lambda}_{\varphi(k,j)} - \lambda_{\varphi(k,j)}  \\
&=  \left( [\mathsf{a}+\omega_k^2 \sin^2(\theta_k)/\theta_k^2 ]^{-1}  - \lambda_{\varphi(k,j)} \right) \\
%&= \left( \frac{1}{\omega_k^2+a + \omega_k^2 (\sin(\theta_k)/\theta_k + 1) (\sin(\theta_k)/\theta_k - 1)}  - \lambda_{\varphi(k,j)} \right) \\
&\le \left( [\mathsf{a}+\omega_k^2 - \omega_k^2 \theta_k^2 /3]^{-1}  - \lambda_{\varphi(k,j)} \right) \quad \text{since $\sin^2(\theta_k)/\theta_k^2 \ge 1-\theta_k^2 / 3$} \\
&\le \lambda_{\varphi(k,j)} \left( [ 1 - \omega_k^2  \theta_k^2 /(3 (\mathsf{a}+\omega_k^2))]^{-1}  - 1 \right)\\
&\le \lambda_{\varphi(k,j)}  \theta_k^2/ 3 (1 -  \theta_k^2/ 3)^{-1}  \\
&\le \lambda_{\varphi(k,j)} 2 \theta_k^2 \quad \text{as required for (E1).}
\end{align*}
For (E2) with $k>1$, by \eqref{Lambda_pimd}, \eqref{lambda_pimd}, and \eqref{eigs_ieqs},  \[
\dfrac{ \bph{\Lambda}_1}{ \bph{\Lambda}_{\varphi(k,j)}} = 1+\frac{\omega_k^2}{\mathsf{a}}  \frac{\sin^2(\theta_k)}{\theta_k^2}   \le 1+\frac{\omega_k^2}{\mathsf{a}}  = \dfrac{\lambda_1}{\lambda_{\varphi(k,j)}} \;.
\] Taking square roots of both sides then gives (E2).
\end{proof}

\begin{proof}[Proof of Theorem~\ref{thm:CONTR_PIMD}]
This proof is very similar to the Proof of Theorem~\ref{thm:CONTR_TPS} with some differences which are highlighted below.  

\medskip

{\em Verify Assumption~\ref{B123} (B1)-(B3).}
Since both $\bph{\mathcal{C}}_{\mathsf{a}}$ 
and $\mathcal{C}_{\mathsf{a}}$ have leading eigenvalue $\mathsf{a}^{-1}$, (B1) holds with $L=1+6 \mathsf{a}^{-1} L_G$.  Similarly, (B3) holds with $K=1/2$ and $A=(1/2)  \beta  \lambda_1^2 M_G^2= (1/2) \beta  \mathsf{a}^{-2} M_G^2$.  Moreover, (B2) holds with $n=2 m_{\ell} d -d = \varphi(m_{\ell}, d)$, since $n+1= \varphi(m_{\ell}+1,1)$, $m_{\ell} \ge  \sqrt{3 L_G/2} (\beta/\pi)$, and $m \ge 2 \pi m_{\ell}$.

\medskip

{\em Verify Conditions~\eqref{ieq:R} \& \eqref{A0A}.} 
By \eqref{Lambda_pimd} and \eqref{eq:om2_k},
\begin{equation} \label{sigma_max_min_pimd}
\sigma_{min} = \mathsf{a}^{1/2} \;, \quad 
\sigma_{max} \le  \beta^{-1} \left( \beta^2 \mathsf{a} + 4 (m_{\ell}-1)^2 \pi^2 \right)^{1/2} \le \sqrt{\mathsf{a} +  6 L_G} \;,
\end{equation}  
since $m_{\ell}-1<\sqrt{3 L_G/2} (\beta/\pi)$.  Moreover, by Lemma~\ref{eigs_PIMD} (E2),  \begin{align}
\frac{\sigma_{max}}{\sigma_{min}} \le \left( 1+ \frac{\omega_{m_{\ell}}^2}{\mathsf{a}} \right)^{1/2} &= \left( 1+ \frac{4 (m_{\ell}-1)^2 \pi^2}{\beta^2 \mathsf{a}}  \right)^{1/2} \le \left( 1+\frac{6 L_G}{\mathsf{a}} \right)^{1/2} \;. \label{sigma_mm_pimd}
\end{align} 
 Furthermore, by Lemma~\ref{eigs_PIMD} (E1), \begin{align}
\tr(\bph{\mathcal{C}}_{\mathsf{a}}) &\le \tr(\mathcal{C}_{\mathsf{a}})+ 2 d 
\sum_{k=1}^{\lceil (m+1)/2 \rceil} ( \bph{\Lambda}_{\varphi(k,1)} -  \lambda_{\varphi(k,1)})  \nonumber \\
&\le  \tr(\mathcal{C}_{\mathsf{a}}) + d \beta^2 = \frac{d}{2 \mathsf{a}}+ \frac{d \beta}{4 \sqrt{\mathsf{a}}} \left( 1 +
\frac{2}{e^{\sqrt{\mathsf{a}} \beta}-1} \right) + d \beta^2 \nonumber \\
&\le 2 d ( \mathsf{a}^{-1}  + \beta^2 )
\label{trace_Ca}
\end{align} where in the last step we used 
$1 + 2/(e^{2 \mathsf{x}}-1) < \mathsf{x} + \mathsf{x}^{-1}$ valid for all $\mathsf{x}>0$.

Let $R_m=8 \sqrt{40} (A + \tr(\bph{\mathcal{C}}))^{1/2} \sigma_{max} L K^{-1/2}$ denote the RHS of \eqref{ieq:R}.  Using $L=1+ 6 \mathsf{a}^{-1} L_G$, $K=1/2$, $A=(1/2) (\beta  / \mathsf{a}^2) M_G^2$, \eqref{sigma_max_min_pimd}, and \eqref{trace_Ca}, \begin{align*}
    R_m^2 &\le 128 \times 40 \mathsf{a} (1 + 6 L_G \mathsf{a}^{-1})^3  ( (1/2) \beta \mathsf{a}^{-2}  M_G^2 + 2 d (\beta^2 + \mathsf{a}^{-1}) ) \\
    &\le 256 \times 20    (1+6 L_G \mathsf{a}^{-1})^3 \left( (1/2)  \beta \mathsf{a}^{-1}M_G^2 + 2 d (\beta^2 \mathsf{a} + 1) \right) \\
    &\le 256 \times 20  (1+\kappa)^3 \left( (1/2) \beta \mathsf{a}^{-1}  M_G^2+ 2 d (\beta^2 \mathsf{a} + 1) \right) = R^2 \;,
\end{align*}
which implies that $R$ defined in \eqref{R_PIMD} satisfies \eqref{ieq:R}.  Moreover, by \eqref{sigma_mm_pimd}, \begin{equation} \label{smmL_A0A_PIMD}
\left(\dfrac{\sigma_{max}}{\sigma_{min}}\right) \L \le  (1+\kappa)^{3/2}
\;.
\end{equation} 
Inserting \eqref{smmL_A0A_PIMD} into \eqref{A0A} gives \eqref{A0A_PIMD}.  

\medskip

{\em Invoke Corollary~\ref{cor:QBHMC}.}
By Corollary~\ref{cor:QBHMC}, as long as $T$ satisfies \eqref{A0A_PIMD}, \eqref{QBpHMC_TPS_1} holds with the dimension-free rate in \eqref{crate_PIMD} and the constants in \eqref{QBpHMC_TPS_Cm} with $\bph{\mathcal{C}}$ replaced with $\bph{\mathcal{C}}_{\mathsf{a}}$.  Moreover, the dimension-dependent constants in \eqref{QBpHMC_TPS_Cm} can be upper bounded by dimension-free constants $C$ and $\epsilon$ given in \eqref{C_PIMD} and \eqref{eps_PIMD}, by using $A=(1/2) \beta M_G^2 \mathsf{a}^{-2}$, \eqref{sigma_max_min_pimd} and \eqref{trace_Ca}.  Thus, \eqref{QBpHMC_TPS} holds for the transition kernel of \eqref{eq:transitionstep_PIMD}.
\end{proof}

\bibliographystyle{amsplain}
\bibliography{bibphmc}

\end{document}